\documentclass[11pt]{amsart}

\usepackage{amssymb,amsmath,amscd,graphicx,latexsym,psfrag,epsfig,color}
\usepackage[cmtip,all]{xy}

\newtheorem{theorem}{Theorem}[section]
\newtheorem{lemma}[theorem]{Lemma}

\theoremstyle{definition}
\newtheorem{definition}[theorem]{Definition}

\newtheorem{xca}[theorem]{Exercise}
\newtheorem{corollary}[theorem]{Corollary}
\theoremstyle{remark}
\newtheorem{remark}[theorem]{Remark}
\newtheorem{proposition}[theorem]{Proposition}

\numberwithin{equation}{section}

\newcommand{\e}{\varepsilon}
\newcommand{\Ue}{U_\varepsilon}
 
\newcommand{\Uq}{U_q}

\newcommand{\mc}{\mathbb{C}}
\newcommand{\mz}{\mathbb{Z}}

\newcommand{\mn}{\mathbb{N}}

\newcommand{\Aa}{{\mathcal A}}

\newcommand{\Dd}{{\mathcal D}}

\newcommand{\Gg}{{\mathcal G}}
\newcommand{\Hh}{{\mathcal H}}
\newcommand{\Ii}{{\mathcal I}}
\newcommand{\Jj}{{\mathcal J}}

\newcommand{\Ll}{{\mathcal L}}

\newcommand{\Oo}{{\mathcal O}}

\newcommand{\Rr}{{\mathcal R}}

\newcommand{\xg}{x_{\chi}}
\newcommand{\yg}{y_{\chi}}
\newcommand{\zg}{z_{\chi}}
\newcommand{\cg}{c_{\chi}}
\newcommand{\g}{\chi}
\def\1{{\rm 1\kern-.8ex 1}}
\newcommand{\ra}{\rightarrow}
\newcommand{\lra}{\longrightarrow}
\newcommand{\fonc}[5]{                     
            \begin{array}{lcll}#1 :& #2 & \lra & #3 \\   %
                         &#4 &\longmapsto & #5          %
            \end{array}}
\newcommand{\fleche}[4]{                     
            \begin{array}{ccl} #1 & \lra & #2 \\   %
                         #3 &\longmapsto & #4          %
            \end{array}}
\newcommand{\psixj}[6]{\left[ \begin{array}{ccc} #1 & #2 & #3 \\ #4 & #5 & #6\end{array}\right]_\e}

\begin{document}

\title[Quantum coadjoint action and the $6j$-symbols of $U_qsl_2$]{Quantum coadjoint action and the $6j$-symbols of $U_qsl_2$}

\author{St\'ephane Baseilhac}
\address{Institut Fourier, Universit\'e de Grenoble, 100 rue des Maths BP 74, 38402 Saint-Martin-d'H\`eres Cedex, France}
\curraddr{Universit\'e Montpellier 2, Institut de Math\'ematiques et de Mod\'elisation, Case Courrier 051, Place Eug\`ene Bataillon, 34095 Montpellier Cedex, France}
\email{stephane.baseilhac@math.univ-montp2.fr}
\thanks{This work was supported by the grant ANR-08-JCJC-0114-01 of the French Agence Nationale de la Recherche.}

\subjclass[2010]{Primary 17B37, 57R56; Secondary 14L24}
\date{January 31, 2010 and, in revised form, XXX.}

\dedicatory{Dedicated to my advisor, Claude Hayat-Legrand, on the occasion of her 65th birthday.}

\keywords{quantum groups, Poisson-Lie groups, coadjoint action, geometric invariant theory, dilogarithm functions, TQFT, invariants of $3$-manifolds}

\begin{abstract} We review the representation theory of the quantum group $U_\e sl_2\mc$ at a root of unity $\e$ of odd order, focusing on geometric aspects related to the $3$-dimensional quantum hyperbolic field theories (QHFT). Our analysis relies on the quantum coadjoint action of De Concini-Kac-Procesi, and the theory of Heisenberg doubles of Poisson-Lie groups and Hopf algebras. We identify the $6j$-symbols of generic representations of $\Ue sl_2\mc$, the main ingredients of QHFT, with a bundle morphism defined over a finite cover of the algebraic quotient $PSL_2\mc/\!/PSL_2\mc$, of degree two times the order of $\e$. It is characterized by a non Abelian $3$-cocycloid identity deforming the fundamental five term relation satisfied by the classical dilogarithm functions, that relates the volume of hyperbolic $3$-polyhedra under retriangulation, and more generally, the simplicial formulas of Chern-Simons invariants of $3$-manifolds with flat $sl_2\mc$-connections.
\end{abstract}

\maketitle




\section{Introduction}
After more than twenty years of outstanding efforts, the geometry of quantum groups, especially the relationships between their representation theories at roots of unity and the underlying Lie groups, remains a prominent matter of quantum topology. It stands, for instance, in the background of the geometric realization problem of the combinatorially defined state spaces of the Reshetikhin-Turaev TQFT (\cite{BHMV}, \cite{Tu}), the deformation quantization of character varieties of $3$-manifolds via skein modules \cite{BFK,PS}, or the asymptotic expansion of quantum invariants (\cite[Ch. 7]{Prob}, \cite{MN}).

In recent years, new and completely unexpected interactions between the {\it non restricted} quantum group $U_q = U_qsl_2\mc$ and $3$-dimensional hyperbolic geometry have been revealed by the volume conjecture \cite{Kh1}, and the subsequent development of the quantum hyperbolic field theories (QHFT) \cite{BB1,BB2,BB3} and quantum Teichm\"uller theory \cite{BL}. The global picture is that many of the fondamental invariants of hyperbolic (ie. $PSL_2\mc$) geometry, like the volume of hyperbolic manifolds, and more generally (some of the) Chern-Simons invariants of $3$-manifolds with flat $sl_2\mc$-connections, should be determined by the semi-classical limits of these quantum algebraic objects. 

Like any topological quantum field theory, the QHFT are symmetric tensor functors defined on a category of $3$-dimensional bordisms. For QHFT, these are equipped with additional structures given by refinements of holonomy representations in $PSL_2\mc$, and links up to isotopy. The QHFT are built from "local" basic datas, universally encoded by the moduli space of isometry classes of ``flattened'' (a kind a framing) hyperbolic ideal $3$-simplices,  and the associativity constraint, or {\it $6j$-symbols}, of the cyclic representations of a Borel subalgebra $U_\e b$ of $U_\e$ at a primitive root of unity $\e$ of odd order. Working with such a moduli space allows one to define the QHFT also for $3$-dimensional bordisms equipped with holonomies having singularities along the links, like cusped hyperbolic manifolds (the link being ``at infinity''), and to get surgery formulas. For cylindrical bordisms the QHFT are related by this way to the local version of quantum Teichm\"uller theory (\cite{BB3},\cite{Bai,BBL}).

Similarly, the Chern-Simons gauge theory for flat $sl_2\mc$-connections, which was originally derived from the integral complex-valued invariant ``volume'' $3$-form of $SL_2\mc$ via secondary characteristic class theory,  defines a functor by means of the same moduli space as the QHFT, using Neumann's simplicial formulas \cite{N} in place of quantum state sums, where a dilogarithm function formally corresponds to the cyclic $6j$-symbols of $U_\e b$ (\cite{BB2}, see also \cite{Ma}). In fact, both maps are $3$-cocycloids in a natural way, the latter on the category of cyclic $U_\e b$-modules with (partially defined) tensor product, and the former on the group $PSL_2\mc$ with discrete topology, via Neumann's isomorphism of $H_3(BPSL_2\mc^\delta;\mz)$ with a certain extension of the Bloch group. 

In this paper we consider this interplay of Abelian vs. non Abelian cohomological structures, which certainly concentrates a key part of the quantization procedure relating the Chern-Simons theory for $PSL_2\mc$ to the QHFT. We describe in detail the non Abelian part of the story, that is, how the simple $\Ue$-modules ``fiber'' over $PSL_2\mc$, and the $6j$-symbols of {\it regular} $U_\e$-modules (rather than the cyclic $U_\e b$-modules). By the way we indicate common features and discrepancies with the $6j$-symbols of the {\it color modules} of the {\it restricted} quantum group $\bar{U}_\e$. 

We point out also that the cyclic $6j$-symbols of $U_\e b$, or (basic) {\it matrix dilogarithms}, coincide with the regular $6j$-symbols of $U_\e$. More precisely, we define a bundle $\Xi^{(2)}$ of regular $\Ue$-modules over a covering of degree $n^2$ ($n$ being the order of $\e$) of a smooth subset of a Poisson-Lie group $H^2$ dual to $PSL_2\mc^2$, endowed with an action of an infinite dimensional Lie group derived from the quantum coadjoint action of De Concini-Kac-Procesi, originally defined for $\Ue$. We have (see Theorem \ref{Rmorphism},  \ref{equivariant} and \ref{reduction} for precise statements):
\begin{theorem} \label{teointro} The regular $6j$-symbols of $U_\e$ and the matrix dilogarithms coincide and define a bundle morphism $\Rr\colon \Xi^{(2)} \lra \Xi^{(2)}$ equivariant under the quantum coadjoint action. 
\end{theorem}
Since the quantum coadjoint action lifts the adjoint action of $PSL_2\mc$ via an unramifield $2$-fold covering $H\ra PSL_2\mc^0$ of the big cell of $PSL_2\mc$, it will follow that $\Rr$ descends to a morphism of a vector bundle of rank $n^2$ over a $2n$-fold covering of the algebraic quotient $PSL_2\mc/\!/PSL_2\mc$. 

The remarkable dependence of the matrix dilogarithms on cross-ratios of $4$-tuples of points on $\mathbb{P}^1$, which was previously known by direct computation and allowed the QHFT to be defined on the moduli space of flattened hyperbolic ideal tetrahedra, is a consequence of Theorem \ref{teointro}. This alternative description is presented in Section \ref{MATDIL}.

In order to achieve our goals we have to review quite a lot of material, starting from the basic properties of $U_\e$ (Section \ref{PBW}), and developing in detail its representation theory and quantum coadjoint action (Section \ref{MODULES} and \ref{QCA}, respectively). The $6j$-symbols of the color modules of $\bar{U}_\e$ and of the regular modules of $\Ue$ are defined in Section \ref{CGO}. Theorem \ref{teointro} is proved in Section \ref{6JMORPH}. There we make a crucial use of fundamental results of Semenov-Tian-Shansky \cite{STS}, Weinstein-Xu \cite{WX}, and Lu \cite{Lu2} on Poisson-Lie groups and their doubles and quantizations.

There should be no obstruction to extend Theorem \ref{teointro} to the quantum groups $U_\e \mathfrak{g}$ of arbitrary complex simple Lie algebras $\mathfrak{g}$. The corresponding quantum coadjoint action theory is described in \cite{DCK}, \cite{DCKP}, \cite{DCP}, and \cite{DCPRR}. Recent works of Geer and Patureau-Mirand \cite{GP} show that the QHFT setup extends to "relative homotopy quantum field theories", including TQFT associated to the categories of finite dimensional weight $U_\e \mathfrak{g}$-modules. A challenging problem is to relate them to the three-dimensional Chern-Simons theory for flat $\mathfrak{g}$-connections. 

Finally, let us note that a similar approach can be used to describe ``holonomy'' $R$-matrices for the regular $U_\e \mathfrak{g}$-modules, in the spirit of \cite{KaRe}.

We will often meet notions from the theory of Poisson-Lie groups. The reader will find the needed material in standard textbooks, like \cite{CP}, \cite{ES} and \cite{KS}.
 
This paper is based on my lectures at the workshop ``Interactions Between Hyperbolic Geometry, Quantum Topology and Number Theory'' held at Columbia University (June 3rd - 19th, 2009). I am grateful to the organizers for their hospitality and wonderful working conditions during my stay. 

\section{The quantum group $U_qsl_2$}\label{PBW}

\subsection{Definition $\&$ PBW basis} We fix our ground ring to be $\mc$, and denote by $q$ a complex number such that $q\ne -1,0,1$. When $q$ is constrained to be a root of unity we denote it by $\e$, and we assume that $\e$ has odd order $n\geq 3$. 

\begin{definition} \label{qg} The quantum group $\Uq = U_qsl_2$ is the algebra generated over $\mc$ by elements $E$, $F$, $K$ and $K^{-1}$, with defining relations $KK^{-1} = K^{-1}K = 1$ and
\begin{equation}\label{nonres}
 KEK^{-1} = q^2 E\ ,\ KFK^{-1} = q^{-2}F\ ,\ \left[E,F\right] = \frac{K-K^{-1}}{q -q^{-1}}.
\end{equation}
The algebra $U_q$ is a \textit{Hopf algebra} with coproduct $\Delta \colon U_q \ra U_q\otimes U_q$, antipode $S\colon U_q\ra U_q$ and counit $\eta\colon U_q \ra \mc$ defined on generators by
\begin{equation}\label{nonres2}\begin{array}{c} \Delta(K) = K \otimes K\\ \Delta(E) = E \otimes 1 + K\otimes E,\ \Delta(F) = 1\otimes F + F \otimes K^{-1}\\S(K) = K^{-1},\ S(E) = -K^{-1}E,\ S(F) = -FK\\
\eta(K)=1,\ \eta(E) = \eta(F) = 0.\end{array}\end{equation}
\end{definition}
To make sense of this definition, let us just recall here that being a Hopf algebra means that $\Delta$ and $\eta$ are morphisms of algebras satisfying the \textit{coassociativity} and \textit{counitality} constraints
\begin{equation}\label{coass} 
 (\Delta \otimes {\rm id})\circ \Delta = ({\rm id}\otimes \Delta)\circ \Delta\quad ,\quad (\eta \otimes {\rm id})\circ \Delta = ({\rm id}\otimes \eta)\circ \Delta = {\rm id}.
\end{equation}
Here the algebra structure of $U_q\otimes U_q$ is by componentwise multiplication, and $\Uq$ is identified with $\Uq\otimes \mc$ in the canonical way. The antipode $S$ is the inverse of the identity for the convolution product, that is, it satisfies
\begin{equation}\label{antipode}
 \mu\circ ({\rm id}\otimes S)\circ \Delta = \mu\circ (S \otimes {\rm id})\circ \Delta = \eta 1,
\end{equation}
where $\mu\colon U_q\otimes U_q\ra U_q$ is the product. In particular, this implies $S(1)=1$, $\eta \circ S = \eta$, $S(xy) = S(y)S(x)$ for all $x$, $y\in U_q$, and \begin{equation}\label{copS} (S\otimes S)\Delta = \tau\circ \Delta \circ S,\end{equation} where $\tau(u\otimes v)=v\otimes u$ is the flip map (\cite[Ch. III]{K}).
\begin{remark} \label{simply-connectedUq} A \textit{simply-connected} version of $U_q$ is often considered in the litterature  (see eg. \cite[\S 9]{DCP}). It is obtained by adding  to the generators a square root of $K$ acting by conjugation on $E$ and $F$ by multiplication by $q$ and $q^{-1}$, respectively.
 \end{remark}
\begin{xca}\label{UEA} Define an algebra $U_q'$ with generators $E$, $F$, $K$ and $K^{-1}$ satisfying the same relations as in $U_q$ except the last one in \eqref{nonres}, and with a further generator $L$ such that
\begin{equation}\begin{array}{c}
 \left[E,F\right]=L\ , \ (q-q^{-1})L=K-K^{-1}\\
 \left[L,E\right] = q(EK+K^{-1}E)\ , \ \left[L,F\right] = -q^{-1}(FK+K^{-1}F).\end{array}
\end{equation}
Show that we have isomorphisms $U_q \cong U_q'$, $U_1' \cong U[K]/(K^2-1)$, $U \cong U_1'/(K-1)$, where $U=Usl_2$ is the universal enveloping algebra of $sl_2$. 
\end{xca}
This exercise shows that, as suggested by the notation, $U_q$ is a genuine deformation depending on the complex parameter $q$ of the universal enveloping algebra of $sl_2$. Like the latter, we can think of $U_q$ as a ring of polynomials (in non commuting variables). In particular, we have the following fundamental result (see \cite[Th. 1.5-1.8]{J} or \cite[Th. VI.1.4]{K} for a proof): 
\begin{theorem} \label{PBWteo} 1) (PBW basis) The monomials $F^tK^sE^r$, where $t$, $r\in \mn$ and $s\in \mz$, make a linear basis of $U_q$.  

2) The algebra $U_q$ has no zero divisors and is given a grading by stipulating that each monomial $F^tK^sE^r$ is homogeneous of degree $r-t$.\end{theorem}

Note that the relations \eqref{nonres} are homogeneous of degree $1$, $-1$ and $0$, respectively. Products of monomials can be written in the basis $\{F^tK^sE^r\}_{r,t\in\mn,s\in \mz}$ by using the two first commutation relations in \eqref{nonres}, together with 
\begin{equation}\label{comEF3}
E^rF^s = \sum_{i=0}^{{\rm min}(r,s)}\ F^{s-i}h_i E^{r-i},\quad r,s\in \mn,
\end{equation}
where the $h_i$ are Laurent polynomials  in $\mc[K,K^{-1}]$ given by  
$$h_i = \left[\!\!\!\begin{array}{c} r\\ i \end{array}\!\!\!\right]\left[\!\!\!\begin{array}{c} s\\ i \end{array}\!\!\!\right] [i]!\ \prod_{j=1}^i \left[K;i+j-(r+s)\right].$$
Here we assume that the product is $1$ for $i=0$, and we put 
\begin{equation}\label{qcom}
\left[K;l\right] = \frac{Kq^{l}-K^{-1}q^{-l}}{q-q^{-1}}.
\end{equation}
We use also the standard notations for $q$-integers, $q$-factorials, and $q$-binomial coefficients 
\begin{equation}\label{qbinomialform} \left[l\right] = \frac{q^l - q^{-l}}{q-q^{-1}}\ ,\ \left[l\right]! = \left[l\right]\left[l-1\right]\ldots \left[1\right]\ ,\ \left[\!\!\!\begin{array}{c} l\\ m \end{array}\!\!\!\right] = \frac{\left[l\right]!}{\left[m\right]!\left[l-m\right]!},\quad l\in \mz,\end{equation}
with $\left[0\right]!=1$ by convention. Note that $[l] = q^{l-1}+\ldots+q^{1-l}$, so $[l] \in \mz[q,q^{-1}]$. Also,
\begin{equation}\label{divide} \left[\!\!\!\begin{array}{c} l\\ 0 \end{array}\!\!\!\right] = \left[\!\!\!\begin{array}{c} l\\ l \end{array}\!\!\!\right] = 1\quad \mbox{and}\quad [l] \ \mbox{divides}\ \left[\!\!\!\begin{array}{c} l\\ m \end{array}\!\!\!\right] \ \mbox{if}\  1<m<l \ \mbox{and}\ l \ \mbox{is odd}.\end{equation}
The elements \eqref{qcom} appear in all computations involving commutators. In the sequel we will use \eqref{comEF3} when $r=1$ and $s$ is arbitrary, or $r$ is arbitrary and $s=1$:
\begin{equation}\label{comEFutile}
EF^s = F^sE + [s]F^{s-1}[K;1-s]\ ,\  FE^r = E^rF - [r]E^{r-1}[K;r-1].
\end{equation}
Another important identity is the \textit{$q$-binomial formula}, which holds for any $u$, $v$ such that $vu = q^{2}uv$ \cite[Ch. VI (1.9)]{K}:
\begin{equation}\label{qbinomial}
 (u+v)^r = \sum_{j=0}^r\left[\!\!\!\begin{array}{c}  r\\ j \end{array}\!\!\!\right]q^{j(r-j)}u^jv^{r-j}.
\end{equation}
 In particular, it implies that $\left[\!\!\!\begin{array}{c} l\\ m \end{array}\!\!\!\right] \in \mz[q,q^{-1}]$.
\begin{xca} \label{Cartan2} (a) Prove \eqref{comEFutile} by induction, using the relations \eqref{nonres}. 

\noindent (b) Show that the second relation in \eqref{comEFutile} follows from the first one by applying the \textit{Cartan automorphism} of the algebra $\Uq$, defined by \begin{equation}\label{Cartan} \omega(E)=F\ ,\  \omega(F)=E\ ,\ w(K)= K^{-1}.\end{equation} 
  \end{xca}
\subsection{Center} Theorem \ref{PBWteo} gives a lot of informations on  the center $Z_q$ of $\Uq$. Indeed, denote by $U_d$ the degree $d$ piece of $\Uq$. For any $u\in U_d$ we have \begin{equation}\label{adKrel} KuK^{-1} = q^{2d}u.\end{equation}
Hence, when $q$ is not a root of unity, $U_d$ is the eigenspace of ad$_K$ with eigenvalue $q^{2d}$. In particular, $Z_q\subset U_0$.

When $q=\e$ is a root of unity of odd order $n\geq 3$, $Z_\e$ is much bigger. Let us show that it is generated by $E^n$, $F^n$ and $Z_\e\cap U_0$. Since the relations \eqref{nonres} are homogeneous, the homogeneous parts of a central element are central. Hence it is enough to prove our claim for a central element $u\in U_d$. By \eqref{adKrel}, if $u \in U_d$ is central then $n$ divides $d$. Now, if $d=ln$ for some $l\in \mz\setminus \{0\}$ with $l>0$ (resp. $l<0$), $U_d$ is spanned by the monomials $F^tK^sE^{t+ln}$ (resp. $F^{t+ln}K^sE^t$). Hence $u = u'E^{ln}$ (resp. $u=F^{ln}u'$) for some $u'\in U_0$. The relations 
\begin{equation}\label{com0} \begin{array}{c} K^lE=q^{2l}EK^l\ ,\ K^lF=q^{-2l}FK^l\\ KE^l = q^{2l}E^lK\ ,\ KF^l = q^{-2l}F^lK\end{array},\quad l\in \mn\end{equation} 
imply that $K^{\pm n}\subset Z_\e$. Together with $\left[n\right]=0$ and \eqref{comEFutile} for $s=n$ and $r=n$, they give also $E^n$, $F^n\subset Z_\e$.  Since $\Ue$ has no zero divisors, we get that any $u\in U_d \cap Z_\e$ can be decomposed as a product of central elements $E^{ln}$ and $F^{ln}$ with a central element in $U_0$.
\begin{definition} The subalgebra $Z_0$ of the center $Z_\e$ of $\Ue$ is generated by $E^n$, $F^n$ and $K^{\pm n}$.\end{definition} The algebra $Z_0$ plays a key role in the structure of $U_\e sl_2$. In particular, the next theorem shows that any element of $Z_\e$ is a root of a monic polynomial with coefficients in $Z_0$. 
\begin{theorem}\label{finite} 1) The algebra $\Ue$ is a free $Z_0$-module of rank $n^3$, with basis the monomials $F^tK^{s}E^r$, $0\leq t,s,r\leq n-1$. 

2) The center $Z_\e$ of $U_\e$ is a finitely generated algebra, and is integrally closed over $Z_0$.  
\end{theorem}
\begin{proof} The first claim follows directly from Theorem \ref{PBWteo} (note in particular that $E^n$, $F^n$ and $K^{\pm n}$ are independent). It implies that $\Ue$ is a noetherian $Z_0$-module, and so the $Z_0$-submodule $Z_\e$ is finitely generated. Hence $Z_\e$ is integrally closed over $Z_0$ \cite[Th. 5.3]{AMcD}. Then, by Hilbert's basis theorem $Z_\e$ is finitely generated as an algebra.
\end{proof}
In another direction, we will see in the next section that $Z_0$ is the basic link between $\Ue$ and the Lie group $PSL_2\mc$. This depends on the fact that it is a (commutative) \textit{Hopf} algebra. To prove this, let us introduce for future reference a prefered set $\{x,y,z\}$ of generators of $Z_0$, given by
\begin{equation}\label{pref}
x= -(\e-\e^{-1})^{n}E^nK^{-n}\ , \ y=(\e-\e^{-1})^{n}F^n\ ,\ z= K^n.
\end{equation}  
We also put
\begin{equation}\label{prefbis} 
e= (\e-\e^{-1})^{n}E^n\ ,\ f = - (\e-\e^{-1})^{n}F^nK^n.\end{equation}
Note that
$$x= T(y)\ ,\ f = T(e),\ x=S(e)\ ,\ y = S(f)$$
where $S$ is the antipode of $\Ue$, and $T$ the \textit{braid group automorphism} of the algebra $\Ue$, defined by
\begin{equation}\label{braid}
T(E)=-FK\ ,\ T(F) = -K^{-1}E\ ,\ T(K)=K^{-1}\ ,\ T(K^{-1})=K.
\end{equation}
\begin{lemma}\label{HopfZ0} The subalgebra $Z_0$ of $\Ue$ is a \textit{Hopf} subalgebra.
\end{lemma}
\begin{proof} We have to show that $\Delta(Z_0)\subset Z_0\otimes Z_0$ and $S(Z_0) \subset Z_0$. Because these maps are morphisms of algebras, it is enough to prove this on the generators $x$, $y$ and $z$ of $Z_0$. We claim that
\begin{equation}\label{nonres3}\begin{array}{c} \Delta(z) = z \otimes z\\ \Delta(x) = x \otimes z^{-1} + 1\otimes x,\ \Delta(y) = 1\otimes y + y \otimes z^{-1}\\S(z) = z^{-1},\ S(x) = -zx,\ S(y) = -yz\\
\eta(z)=1,\ \eta(x) = \eta(y) = 0.\end{array}\end{equation}
The proof of the formulas for $S$ is immediate. As for $\Delta$, they are consequences of \eqref{divide}, the $q$-binomial identity \eqref{qbinomial}, and the fact that $[n]=0$ when $q=\e$.
\end{proof}
So far, we have only considered in \eqref{adKrel} the invariance of the center of $\Uq$ under conjugation by $K$. In order to describe the effect of the conjugation action by $E$ and $F$, let us introduce the automorphisms $\gamma_l$ of $\mc[K,K^{-1}]$ given by \begin{equation}\label{autl} \gamma_l(K) = q^{l}K,\quad  l\in \mz.                                                           \end{equation}
By \eqref{nonres}, for all $h\in \mc[K,K^{-1}]$ we have $$hE=E\gamma_2(h) \quad \mbox{and}\quad hF=F\gamma_{-2}(h).$$ It is easy to check that $u=\sum_{i\geq 0} F^ih_iE^i \in U_0$ satisfies $uE=Eu$ and $uF=Fu$ in $\Uq$ if and only if for all $i$,
\begin{equation}\label{rec} h_i-\gamma_{-2}(h_i) = [i+1][K;-i]h_{i+1}.
\end{equation}
Since $\mc[K,K^{-1}]$ is an integral domain, $h_{i+1}$ is thus determined inductively by $h_0$ if $[j+1]\ne 0$ for all $0\leq j\leq i$. This condition is satisfied for all $i$ when $q$ is not a root of unity, and for $i<n-1$ when $q=\e$. Hence the projection map 
\begin{equation}\label{HCdef} \fonc{\pi}{U_0}{\mc[K,K^{-1}]}{\sum_{r=0}^{\infty}F^rh_rE^r}{h_0}
\end{equation}
is injective over $Z_q\subset U_0$ when $q$ is not a root of unity, and injective over $Z_\e \cap U_0'$ when $q=\e$, where
\begin{equation}\label{resset} U_0' = \left\lbrace \sum_{i=0}^{n-1} F^ih_iE^i\ |\ h_0,\ldots, h_{n-1}\in \mc[K,K^{-1}]\right\rbrace \subset U_0.
\end{equation}
In fact, Ker$(\pi) = FU_0E$ is a two-sided ideal in $U_0$, so $\pi$ is a homomorphism of algebras, the \textit{Harish-Chandra homomorphism}. It is easy to check that $Z_\e$ is generated by $F^n$, $E^n$ and $Z_\e \cap U_0'$. On another hand, \eqref{rec} is solved by  
$$h_0 = \frac{qK+q^{-1}K^{-1}}{(q-q^{-1})^2}\ ,\ h_1=1\ , \ h_r =0\quad \mbox{if} \ r\geq 2,$$
which defines the \textit{Casimir element}
\begin{equation}\label{Casimir} \Omega = FE + \frac{qK+q^{-1}K^{-1}}{(q-q^{-1})^{2}}\ \in Z_q.\end{equation}
Note that (see \eqref{autl})
$$\gamma_{-1}\circ \pi(\Omega) = \frac{K+K^{-1}}{(q-q^{-1})^2}.$$ 
We will see in Section \ref{MODULES} that $Z_q$ is actually generated by $\Omega$, and $Z_\e\cap U_0'$ by $\Omega$ and $K^{\pm n}$ (see Theorem \ref{moduleteo1}; for an alternative approach based on induction on $i$ in \eqref{rec}, see \cite[Prop. 2.20]{J}). More precisely:
\begin{theorem} \label{amainteo} (1) If $q$ is not a root of unity, $Z_q$ is a polynomial algebra over $\mc$ generated by $\Omega$, and $\gamma_{-1}\circ \pi:Z_q \ra \mc[K+K^{-1}]$ is an isomorphism.

(2) If $q=\e$ is a root of unity of odd order $n\geq 3$, then $Z_\e$ is the algebra over $\mc$ generated by $E^n$, $F^n$, $K^{\pm n}$ and $\Omega$ with relation
\begin{equation}\label{ext0}
 \prod_{j=0}^{n-1}\left( \Omega - c_j\right) = E^nF^n+\frac{K^n+K^{-n}-2}{(\e-\e^{-1})^{2n}}
\end{equation}
where $$c_j =\frac{\e^{j+1}+\e^{-j-1}}{(\e-\e^{-1})^{2}}.$$
\end{theorem}
Statement (2) is part of \cite[Th. 4.2]{DCK}. It shows that $Z_\e$ is a finite extension of $Z_0$ built on the image $\mc[\Omega]$ of the injective homomorphism
\begin{equation}\label{injfix}
(\gamma_{-1}\circ \pi)^{-1}\colon \mc[K+K^{-1}] \longrightarrow Z_\e.
\end{equation}
The domain $\mc[K+K^{-1}]$ is the fixed point set of $\mc[K,K^{-1}]$ under the involution $s(K^{\pm 1}) = K^{\mp 1}$, which can be identified with a generator of the Weyl group of $PSL_2\mc$. 
\section{The quantum coadjoint action for $\Ue$}\label{QCA}
Consider the set Spec$(Z_\e)$ of algebra homomorphisms from the center $Z_\e$ of $\Ue$ to $\mc$. An element of Spec$(Z_\e)$ is called a \textit{central character} of $\Ue$. The set Spec$(Z_\e)$ has a very rich Poisson geometry related to the adjoint action of $PSL_2\mc$, that we are going to describe. It will be used in Section \ref{MODULES}.
\smallskip

Since $Z_\e$ is finitely generated, by Hilbert's nullstellensatz Spec$(Z_\e)$ is an affine algebraic set. The inclusion $Z_0\subset Z_\e$ induces a regular (restriction) map 
\begin{equation}\label{tau} \tau\colon {\rm Spec}(Z_\e) \lra {\rm Spec}(Z_0).
\end{equation}
We have seen in Proposition \ref{finite} that $Z_\e$ is integrally closed over $Z_0$. Hence $\tau$ is \textit{finite}: any $y\in {\rm Im}(\tau)$ has a finite number of preimages. From this property it follows that $\tau$ is surjective \cite[\S 5.3]{Sh}.
\subsection{A Poisson-Lie group structure on Spec$(Z_0)$} \label{PLsection} We need some explicit formulas. Denote by $T$, $U_\pm$ and $PB_{\pm}$ the subgroups of $PSL_2\mc = SL_2\mc/(\pm 1)$ of diagonal, upper/lower unipotent and upper/lower triangular matrices up to sign, with their natural structures of affine algebraic groups induced by the adjoint action on $sl_2\mc$. The action
$$\left(\begin{array}{cc} a & b\\ c & d\end{array}\right)\colon w\mapsto \frac{aw+b}{cw+d},\ w\in \mathbb{P}^1$$
identifies $PSL_2\mc$ with the group ${\rm Aut}(\mathbb{P}^1)$ of automorphisms of the Riemann sphere, and $T$ with the subgroup of maps $w\mapsto zw$, $z\in \mc^*$. Consider the group 
$$H = \{ (tu_+,t^{-1}u_-)\ | t\in T, u_\pm \in U_\pm \} \subset PB_{+}\times PB_{-}.$$
By evaluating on the generators $x$, $y$ and $z$ of $Z_0$ (see \eqref{pref}) we get a map 
\begin{equation}\label{psiprime}\fonc{\psi'}{{\rm Spec}(Z_0)}{\mc^2\times \mc^*}{g}{(x_g,y_g,z_g).}\end{equation}
For all $g\in {\rm Spec}(Z_0)$, consider the automorphisms of $\mathbb{P}^1$ given in $\psi'$-coordinates by
$$\psi_1(g)\colon w \mapsto z_gw-x_gz_g\ ,\ \psi_2(g)\colon w\mapsto w/(-z_gy_gw+z_g),\quad w\in \mathbb{P}^1.$$ 
Under the above isomorphism of ${\rm Aut}(\mathbb{P}^1)$ with $PSL_2\mc$, they correspond respectively to
$$\pm \left(\begin{array}{cc} \sqrt{z_g} & -x_g\sqrt{z_g}\\ 0 & 1/\sqrt{z_g}\end{array}\right)\in PB_{+}\quad ,\quad \pm \left(\begin{array}{cc} 1/\sqrt{z_g} & 0\\ -y_g\sqrt{z_g} & \sqrt{z_g}\end{array}\right)\in PB_{-}$$
for any choice of square root $\sqrt{z_g}$. Define 
$$\fonc{\psi}{{\rm Spec}(Z_0)}{H}{g}{(\psi_1(g),\psi_2(g)).}$$
Clearly, $\psi$ is an isomorphism of algebraic varieties.  It gives $H$ a Poisson-Lie bracket as follows. By using the PBW basis $\{F^tK^sE^r\}_{r,t\in\mn,s\in \mz}$ we can identify $\Ue$ as a linear subspace of $U_q$, considered as a family of algebras over $\mc$ with varying parameter $q$. Hence, for any given element $u\in \Ue$ we can specify a {\it lift} $\tilde{u}\in \Uq$, such that $$u = \tilde{u}\ {\rm mod}(q^n-q^{-n})U_q.$$ 
We can write this as $u = \lim_{q\rightarrow \e} \tilde{u}$. Then, for all $a\in Z_0$ and $u\in U_\e$, put
$$D_a(u) = \lim_{q\rightarrow \e}\frac{\left[\tilde{a},\tilde u\right]}{n(q^n-q^{-n})}.$$
Note that
\begin{equation}\label{der0} D_a(u)= \frac{1}{(\e-\e^{-1})}\lim_{q\rightarrow \e}\left[\frac{\tilde{a}}{n[n]},\tilde u\right] = \frac{1}{(\e-\e^{-1})^n}\lim_{q\rightarrow \e}\left[\frac{\tilde{a}}{[n]!},\tilde u\right]\end{equation}
where we use the formula $[n-1]! = n(\e-\e^{-1})^{1-n}$.
\begin{proposition}\label{center} \textit{(i) The map $\psi \colon {\rm Spec}(Z_0) \ra H$ is an isomorphism of algebraic groups.}

\textit{(ii) The maps $D_a\colon \Ue \ra \Ue$ are well-defined derivations of $U_\e$ preserving $Z_0$ and $Z_\e$. At the generators $e=-xz$ and $z$ in \eqref{prefbis} we have:
\begin{equation}\label{derform1}
D_e(F) = \frac{(\e-\e^{-1})^{n-1}}{n}[K;1]E^{n-1}\ ,\ D_e(K^{\pm 1}) = \mp \frac{(\e-\e^{-1})^n}{n}K^{\pm 1}E^n\end{equation}
\begin{equation}\label{derform2} D_z(E) = \frac{1}{n}K^{n}E\ ,\  D_z(F) = -\frac{1}{n}FK^{n}.\end{equation}
where $[K;l]$ is given by \eqref{qcom} with $q=\e$.}
\end{proposition}
\begin{proof} (i) This follows from classical duality arguments in Hopf algebra theory (see \cite[$\S$ 9]{Mo}). Because $Z_0$ is a commutative Hopf algebra (Lemma \ref{HopfZ0}), the algebraic set Spec$(Z_0)$ has a canonical group structure identifying $Z_0$ with the algebra of regular functions on Spec$(Z_0)$. It is defined dually by stipulating that for any $u\in Z_0$ and $g$, $h\in {\rm Spec}(Z_0)$,
\begin{equation}\label{group0} u(gh) = \mu\circ \Delta(u)(f,g)\ ,\ u(g^{-1}) = S(u)(g)\ ,\ u(\1) = \eta(u).
 \end{equation}
Here, $\1\in {\rm Spec}(Z_0)$ is the identity and $\mu$ the product of $Z_0$. The associativity of the product of Spec$(Z_0)$ follows from the coassociativity of $\Delta$ (see \eqref{coass}). The inverse is well-defined by $u(\1) = \eta(u)$ and the computation
$$u(gg^{-1}) = \mu\circ ({\rm id}\otimes S)\circ \Delta(u)(g,g) = \eta(u)1(g) = \eta(u),$$ 
where $1\in Z_0$ is identified with the constant function on Spec$(Z_0)$ with value $1$, and we use \eqref{antipode} in the second equality. By \eqref{nonres3} and \eqref{group0} we get in the $\psi'$-coordinates $(x_g,y_g,z_g)$ of the point $g$ of Spec$(Z_0)$:
\begin{equation}\label{group}
\begin{array}{c}
x_{gh} = x_h+x_gz_h^{-1}\ ,\ y_{gh} = y_h+y_gz_h^{-1}\ ,\ z_{gh} = z_gz_h\\
x_{g^{-1}} =-z_gx_g\ ,\ y_{g^{-1}} = -y_gz_g\ ,\ z_{g^{-1}} = z_g^{-1}\\
x_\1 = 0\ ,\ y_\1 = 0\ ,\ z_\1 = 1. 
\end{array}
\end{equation}
Now we have\begin{align}
\psi_1(g)\psi_1(h)(w) = & z_g(z_hw-z_hx_h)-z_gx_g \notag \\
= & z_gz_hw -z_gz_h(x_h+x_gz_h^{-1}),\notag
\end{align}
which is equal to $\psi_1(gh)$ by \eqref{group}. A similar computation gives $\psi_2(g)\psi_2(h)(w) = \psi_2(gh)$. Hence $\psi$ is a group homomorphism.

\smallskip

(ii) The maps $D_a$ are well-defined because $a\in Z_\e$ implies $[\tilde{a},\tilde{u}] \in (q^n-q^{-n})U_q$. They are derivations of $\Ue$ because the commutator $[\ ,\ ]$ of $U_q$ satisfies the Leibniz rule. Clearly, they map $Z_\e$ to $Z_\e$, and also $Z_0$ to $Z_0$ since for all $u\in Z_0$, $D_a(u)$ is a sum of monomials in $Z_0$. 

All this can also be check by direct computation, as follows. By \eqref{comEF3}, the products of basis elements of $\Uq$ read as
$$(F^tK^sE^r)(F^cK^bE^a) = \sum_{i=0}^{{\rm min}(r,c)}q^{-2((r-i)b+s(c-i))} F^{c+t-i}h_i K^{s+b}E^{r+a-i}.$$
Suppose that $r$, $s$ and $t$ are (possibly $0$) multiples of $n$, say $r=r'n$, $s=s'n$ and $t=t'n$. Then, all the $h_i$ except $h_0=1$ are divided by $[ln]$ for some $l$, and so vanish at $q=\e$. Since the coefficient of $F^{c+t}K^{s+b}E^{r+a}$ in the commutator $[F^{t}K^{s}E^{r},F^cK^bE^a]$ is 
$$q^{-2n(r'b+s'c)} - q^{-2n(s'a+bt')} =  
 q^{-n(b(r'+t')+s'(a+c)} [n(s'(a-c)+b(t'-r'))],$$
and $\lim_{q\ra \e} [nl]/[n] = l$ for any integer $l$, the limit \eqref{der0} is well-defined. 

The formula for $D_e(F)$ follows from a straightforward computation using \eqref{comEFutile} and \eqref{com0}. Also we have 
\begin{align}\label{eqconv}
D_e(K^{\pm 1}) = & \lim_{q\ra \e}\frac{[(q-q^{-1})^nE^n,K^{\pm 1}]}{n(q^n-q^{-n})}\notag\\
       = & \lim_{q\ra \e}\frac{(q-q^{-1})^n}{n}\ \frac{q^{\mp 2n}-1}{q^n-q^{-n}}K^{\pm 1}E^n\notag\\
       = & \mp \frac{(\e-\e^{-1})^n}{n}K^{\pm 1}E^n.\notag
  \end{align}
We get $D_z$ by similar computations. 
\end{proof}
For all $u$, $v\in Z_0$ and $f\in {\rm Spec}(Z_0)$, put
$$\{u,v\}(f) = D_u(v)(f).$$
By Proposition \ref{center}, for all $a\in Z_0$, $D_a\colon Z_0\ra Z_0$ defines an algebraic vector field on Spec$(Z_0)$, and $\{\ ,\ \}$ is a bivector satisfying the Jacobi identity. Hence:
\begin{corollary} \label{PoissonH} \textit {The bivector $\{\ ,\ \}$ is a Poisson bracket on Spec$(Z_0)$.}
\end{corollary}
\begin{xca}\label{formPoisson} (a) Check that  
$$\begin{array}{c}\{e,x\} = -zx^2\ ,\ \{e,y\} = z-z^{-1}\ ,\ \{e,z\} = xz^2\\ \{f,x\} = z^{-1}-z\ ,\ \{f,y\} = y^2z\ ,\ \{f,z\} = -yz^2\\ \{y,x\} = 1-xy-z^{-2}\ ,\ \{z,x\} = zx\ ,\ \{z,y\} = -yz.\end{array}$$
(b) Deduce that the Poisson bracket $\psi_*\{\ ,\ \}$ on $H$ does not depend on the order $n$ of $\e$. 

\noindent (c) Show that for any algebra automorphism $\phi$ of $\Ue$ we have \begin{equation}\label{autder} \phi D_a \phi^{-1}(u) = D_{\phi(a)}(u).                                                                                                                                                                                                                                                                                  \end{equation}
Deduce that the group of algebra automorphisms of $\Ue$ induces a group of Poisson automorphisms of $({\rm Spec}(Z_0),\{\ ,\ \})$. By using the braid group automorphism \eqref{braid}, compute the value of the derivation $D_f$ on $E$, $F$ and $K^{\pm 1}$.

(d) Show that the derivations $D_{\tilde a}$ obtained from the $D_a$ by allowing the lift $\tilde{a}$ to be arbitrary generate a Lie algebra $\mathcal{L}$. Show that  $\mathcal{L}$ fits into an exact sequence of Lie algebras $0\longrightarrow \mathcal{L}^0 \longrightarrow \mathcal{L} \longrightarrow \mathcal{L}'\longrightarrow 0$, where $\mathcal{L}^0$ (resp. $\mathcal{L}'$) is the Lie algebra of inner derivations of $\Ue$ (resp. of derivations of $Z_0$ induced by $\{\ ,\ \}$) . 

\end{xca}
\begin{remark}\label{dualLPG} \textit{(Semi-classical geometry, I)} By Exercise \ref{formPoisson} the Poisson bracket $\psi_*\{\ ,\ \}$ is canonically associated to $\Uq$. Conversally, $(H,\psi_*\{\ ,\ \})$ is \textit{dual} to the standard Poisson-Lie structure on $PSL_2\mc$, and the {\it rational quantization} of the latter is $\Uq$ (\cite[\S 7]{DCKP}, \cite[\S 11, 14 $\&$ 19]{DCP}).
\end{remark}
Denote by $PSL_2\mc^0=U_-TU_+$ the {\it big cell} of $PSL_2\mc$; it is the open subset consisting of matrices up to sign with non vanishing upper left entry. We have an unramified $2$-fold covering 
\begin{equation}\label{tobigcell}\begin{array}{ccll} \sigma:& H & \longrightarrow & PSL_2\mc^0\\
   & (t^{-1}u_-,tu_+) & \longmapsto & u_-^{-1}t^2u_+.
  \end{array}\end{equation}
Consider the matrices
\begin{equation}\label{fundrep}
\left(\begin{array}{cc} 0 & 1\\ 0 & 0\end{array}\right)\ ,\ \left(\begin{array}{cc} 0 & 0\\ 1& 0\end{array}\right)\ ,\ \left(\begin{array}{cr} 1 & 0\\ 0 & -1\end{array}\right).
\end{equation}
They correspond to the Chevalley generators of $sl_2\mc$ in its fundamental representation. Since $\sigma$ is an unramified covering, the associated (complex) left invariant vector fields on $PSL_2\mc$ lift to vector fields $\underline{e}$, $\underline{f}$ and $\underline{h}$ on $H$. By using Proposition \ref{center} we consider them as vector fields on Spec$(Z_0)$. 

Recall the generators $e$, $f$ and $z$ of $Z_0$, which define functions on Spec$(Z_0)$. Denote by $\mathfrak{g}$ (resp. $\tilde{\mathfrak{g}}$) the Lie algebras of vector fields on Spec$(Z_0)$ generated by $\underline{e}$ and $\underline{f}$ (resp. by $D_e$ and $D_f$). 
\begin{proposition}\label{fields} \textit{We have
\begin{equation}\label{vect} D_e=z\underline{f}\ ,\ D_f=-z\underline{e}\ ,\ D_z=z\underline{h}/2
\end{equation}
and 
\begin{equation}\label{vect2}[D_e,D_f] = z^2\underline{h}-z^2x\underline{e}+z^2y\underline{f}.\end{equation}
Hence the linear spans of $\mathfrak{g}$ and $\tilde{\mathfrak{g}}$ coincide at every point of Spec$(Z_0)$.}
\end{proposition}
\begin{proof} Since all vector fields are algebraic, it is enough to check \eqref{vect} on functions on Spec$(Z_0)\cong H$ lifting functions on $PSL_2\mc^0$, and we can restrict to coordinates of the map $\sigma$. Put
 $$\ u_- = \pm\left(\begin{array}{rc} 1 & 0\\ -y & 1\end{array}\right)\ ,\ u_+ = \pm\left(\begin{array}{cr} 1 & -x\\ 0& 1\end{array}\right)\ ,\ t = \pm\left(\begin{array}{cc} \sqrt{z} & 0\\ 0 & 1/\sqrt{z}\end{array}\right)$$
and
$$M = \sigma(tu_+,t^{-1}u_-) = u_-^{-1}t^2u_+ = \pm\left(\begin{array}{rc} z & -zx\\ zy & -zxy+z^{-1}\end{array}\right).$$
By Exercise \ref{formPoisson} (a) we have
$$D_eM  = \pm\left(\begin{array}{cc} z^2x & 0\\ z^2xy+z^2-1 & -z^2x\end{array}\right).$$
Let $a\in sl_2\mc$. Denote by $a_*$ the left invariant vector field  on $PSL_2\mc$ associated to $a$, and $\underline{a}$ its pull-back to Spec$(Z_0)$ via $\sigma$. By definition, $\underline{a}(M) = a_*(M)$, and for all $g\in PSL(2,\mc)^0$, 
$$a_*(M)(g) = \frac{{\rm d}}{{\rm dt}}\left(M(ge^{at})\right)_{t=0} = \frac{{\rm d}}{{\rm dt}}\left(R_{e^{at}}^*M\right)_{t=0}(g),$$ 
where $R_{e^{at}}$ denotes right translation by $e^{at}$. On another hand, for any vector field $X$ on $PSL_2\mc$ we have $R_{h*}X = {\rm Ad}_{h^{-1}}(X)$. Since the adjoint representation is faithful, dually $(R_{h}^*\rho)(g) = {\rm Ad}_{h}(\rho(g))$ for any linear representation $\rho$ of $PSL_2\mc$. Hence 
$$a_*(M)(g) = {\rm ad}_{a}(M)(g) = [a,M(g)].$$
It is immediate to check that $$[\left(\begin{array}{cc} 0 & 0\\ z& 0\end{array}\right),M] = D_eM.$$
Hence $D_e=z\underline{f}$. The formulas for $D_f$ and $D_z$ can be proved in the same way. We obtain \eqref{vect2} by using $[\underline{e},\underline{f}] =\underline{h}$, the formulas in Exercise \ref{formPoisson} (a), and the fact that 
$$[fU,gV] = fg[U,V]+fU(g)V-gV(f)U$$
for any vector fields $U$, $V$ and functions $f$, $g$.  
\end{proof}
\subsection{Flows on $\Ue$} The infinitesimal action of the Lie algebra $\tilde{\mathfrak{g}}$ generated by $D_e$ and $D_f$ can be integrated to an infinite dimensional group of automorphisms of the algebra $\Ue$. In fact, in order to make sense of this we have to allow holomorphic series in the generators of $Z_0$ as coefficients. 

More precisely, denote by $\hat{Z}_0$ the algebra of power series in the generators $x$, $y$, $z$ and $z^{-1}$ of $Z_0$ which converge to a holomorphic function for all values of $(x,y,z)$ in $\mc^2\times \mc^*$. Set
\begin{equation}\label{extalg}\hat{U}_\e = \Ue \otimes_{Z_0} \hat{Z}_0\ ,\ \hat{Z}_\e = Z_\e \otimes_{Z_0} \hat{Z}_0.
\end{equation}

\begin{proposition} \label{qca} \textit{For all values of $t\in \mc$ the series $\exp(tD_e)$, $\exp(tD_f)$ and $\exp(tD_z)$ converge to automorphisms of $\hat{U}_\e$ preserving $\hat{Z}_0$ and $\hat{Z}_\e$. }
\end{proposition}
\begin{proof} By \eqref{derform2} the statement is clearly true for $\exp(tD_z)$. To conclude it is enough to check that the series $\exp(tD_e)$ when applied to the generators $K$ and $F$ converge to an element of $\hat{U}_e$; the result for $\exp(tD_f)$ follows from this by using $T(e) = f$ (see \eqref{prefbis}) and Exercise \ref{formPoisson} (c). Now, by \eqref{derform1}-\eqref{derform2} we have
$$\exp(tD_e)K = \exp(-te/n)K$$  
and 
$$(D_e)^l(F) = -\frac{(\e-\e^{-1})^{n-2}}{n^l} e^{l-1}((-1)^{l}K\e +K^{-1}\e^{-1})E^{n-1}.$$
Hence
$$\exp(tD_e)F = F -(\e-\e^{-1})^{n-2}\left(\frac{e^{-te/n}-1}{e}K\e+\frac{e^{te/n}-1}{e}K^{-1}\e^{-1}\right)E^{n-1}.$$
The stability of $\hat{Z}_0$ and $\hat{Z}_\e$ follows from that of $Z_0$ and $Z_\e$ under the derivations $D_a$ (see Proposition \ref{center}). 
\end{proof}
\begin{definition} \label{qcagroup} The subgroup $\mathcal{G}$ of ${\rm Aut}(\hat{U}_\e)$ is generated by the $1$-parameter groups $\exp(tD_e)$ and $\exp(tD_f)$, $t\in \mc$.
\end{definition}
Since $\Gg$ leaves $\hat{Z}_0$ and $\hat{Z}_\e$ invariant, it acts dually by holomorphic transformations on the varieties Spec$(Z_0)$ and Spec$(Z_\e)$:
$$\forall g\in \Gg,u\in \hat{Z}_\e,\chi\in {\rm Spec}(Z_\e),\quad (g.u)(\chi) = u(g^{-1}.\chi).$$
By Proposition \ref{fields}, the $\Gg$-action on Spec$(Z_0)$ lifts the conjugation action of $PSL_2\mc$. Also, it follows from \eqref{vect2} that the $\Gg$-orbits are both open and closed in the preimage of any conjugacy class, and so are connected. Since any element of $PSL_2\mc$ is conjugate to one in a given Borel subgroup, any conjugacy class in $PSL_2\mc$ intersects $PSL_2\mc^0$ in a non empty smooth connected variety. Hence we have \cite[Prop. 6.1]{DCKP}:
\begin{theorem}\label{QCAteo} (i) For any conjugacy class $\Gamma$ in $PSL_2\mc$, the connected components of the variety $\sigma^{-1}\Gamma$ are orbits of $\Gg$ in Spec$(Z_0)$. 

(ii) The set $\sigma^{-1}(\{{\rm id}\})$ coincides with the fixed point set of $\Gg$ in Spec$(Z_0)$.
\end{theorem}
\begin{remark} \textit{(Semi-classical geometry, II)} (i) By construction, the orbits of the $\Gg$-action on Spec$(Z_0)$ are symplectic leaves for the Poisson bracket $\{\ ,\ \}$. A result of M. Semenov-Tian-Shansky implies that they coincide under $\psi$ with the orbits of the dressing action of $PSL_2\mc$ on $H$ (\cite[Ch. 1.3]{KS}, \cite[Ch. 1.5]{CP}) (see Remark \ref{dualLPG}). However there does not seem to have any homomorphism $\Gg \rightarrow PSL_2\mc$ relating these two actions on $H$.

(ii) The action of $\Gg$ on Spec$(Z_0)$ is called the {\it quantum coadjoint action} \cite{DCK}, because it coincides at the tangent space of fixed points with the coadjoint action of $PSL_2\mc$ on its dual Lie algebra. 
\end{remark}
\section{Representation theory of $\Ue$}\label{MODULES}
In this section we use the quantum coadjoint action of $\Gg$ on Spec$(Z_0)$, and its extension to $\hat{U}_e$, to describe the representation theory of $U_\e$ in geometric terms. The main results are Theorem \ref{bundle} and Theorem \ref{bundle2}.
\smallskip

To begin with, recall that two $\Ue$-modules $V$ and $W$ are \textit{isomorphic} if there exists a linear isomorphism $\theta:V\ra W$ commuting with the action of $\Ue$. A $\Ue$-module is \textit{simple} if it has no proper non trivial submodule. In what follows, all modules are \textit{left} modules.

It is classical that any simple $\Ue$-module is a finite dimensional vector space \cite[p. 339]{CP}, and that the action of the center $Z_\e$ is by scalars. More precisely: \begin{lemma} \label{scacentre} Any central element $u\in Z_q$ acts on a finite dimensional simple $\Uq$-module $V$ by multiplication by a scalar.
\end{lemma} 
\begin{proof} Since $\mc$ is algebraically closed and $V$ is finite dimensional, the linear operator $u_V \in{\rm End}(V)$ associated to $u$ has an eigenvalue $\lambda$.  Because $u$ is central, $u_V$ commutes with the action of $\Uq$. Hence the kernel of $u_V-\lambda .{\rm id}$ must be the whole of $V$, for otherwise it would be a proper non zero submodule.\end{proof}
By the lemma we can associate to any simple $\Ue$-module $V$ the central character $\chi_V\colon Z_\e\ra \mc$ in ${\rm Spec}(Z_\e)$ such that $uv = \chi_V(u)v$ for all $u\in Z_\e$ and $v\in V$. Clearly any two isomorphic $\Ue$-modules define the same central character. Hence, denoting by ${\rm Rep}(U_\e)$ the set of isomorphism classes of simple $U_\e$-modules, we have a {\it central character map}
\begin{equation}\label{Xi}\begin{array}{lcll} \Xi\colon& {\rm Rep}(\Ue) & \longrightarrow & {\rm Spec}(Z_\e)\\
& V & \longmapsto & \chi_V. \end{array}\end{equation}   
For any $\chi\in {\rm Spec}(Z_\e)$, consider the two-sided ideal $\mathcal{I}^{\chi}$ of $\Ue$ generated by its kernel, ${\rm Ker}(\chi) =\{z-\chi(z) .{\rm id} |z\in Z_\e\}$. Define an algebra
\begin{equation}\label{redalg} \Ue^{\chi} = \Ue/\mathcal{I}^{\chi}.
\end{equation}
 By Theorem \ref{finite} and the fact that the degree of a homogeneous element of $Z_\e$ is always a multiple of $n$, the algebra $\Ue^{\chi}$ is finite dimensional over $\mc$ and non zero. It is also a $\Ue$-module by the left regular representation on cosets, and clearly any simple submodule of $\Ue^{\chi}$ is in $\Xi^{-1}(\chi)$. This proves:
\begin{lemma} The central character map $\Xi$ is surjective. 
\end{lemma}
In order to clarify the properties of $\Xi$ we are going to describe the set Rep$(U_\e)$ explicitly.
\subsection{The simple $\Ue$-modules}\label{simpledef} Consider the following two families of $\Ue$-modules:
\begin{itemize}
\item the modules $V_{r}^\pm$ of dimension $r+1$, where $0\leq r\leq n-1$, with a basis $v_0,\ldots ,v_r$ such that  
$$Kv_j = \pm \e^{r-2j}v_j\quad ,\quad F v_j = \left\lbrace \begin{array}{ll} v_{j+1}, & \mbox{if} \ j<r\\ 0, & \mbox{if} \ j=r\end{array}\right.$$
$$E v_j = \left\lbrace \begin{array}{ll}\pm [j][r-j+1] v_{j-1}, & \mbox{if} \ j>0\\ 0,& \mbox{if} \ j=0.\end{array}\right.$$ 
\item the modules $V(\lambda,a,b)$ of dimension $n$, where $\lambda\in \mc^*$ and $a$, $b\in\mc$, with a basis $v_0,\ldots ,v_{n-1}$ such that     
$$Kv_j =  \lambda\varepsilon^{-2j}v_j\quad ,\quad F v_j = \left\lbrace \begin{array}{ll} v_{j+1}, & \mbox{if} \ j<n-1\\ bv_0, & \mbox{if} \ j=n-1\end{array}\right.$$
$$E\cdot v_j = \left\lbrace \begin{array}{ll} a v_{n-1}, & \mbox{if} \ j=0\\ \left(ab+[j]\frac{\lambda\e^{1-j}-\lambda^{-1}\e^{j-1}}{\e-\e^{-1}}\right)v_{j-1},& \mbox{if} \ j>0.\end{array}\right.$$ 
\end{itemize}
That these formulas actually define actions of $\Ue$ is a direct consequence of \eqref{nonres} and \eqref{comEFutile}. We will see in Section \ref{CGO} that the modules $V_r^-$ and $V_r^+$ are essentially equivalent (we have an isomorphism $V_r^- = V_0^-\otimes V_r^+$). The two families of modules $V_r^\pm$ and $V(\lambda,a,b)$ intersect exactly at $V_{n-1}^\pm = V(\pm \e^{-1},0,0)$. The action of $\Ue$ on the modules $V_r^\pm$ is often expressed in the litterature in a different form, by using the ``balanced'' basis $m_j = v_j/[j]!$ rather than the $v_j$. \smallskip

\subsubsection{ Highest weight $\Ue$-modules} The modules $V_r^\pm$ and $V(\lambda,0,b)$ have the important common property to be highest weight modules. 

Recall that a $\Ue$-module is a \textit{highest weight module} of highest weight $\lambda\in \mc^*$ if it is generated by a non zero vector $v$ such that $Ev=0$ and $Kv=\lambda v$; $v$ is called a \textit{highest weight vector}. It is canonical up to a scalar factor, since $K$ is diagonal in the basis $F^iv$. In fact, there is a \textit{universal} $\Ue$-module of highest weight $\lambda$, the (infinite-dimensional) \textit{Verma module}  $M(\lambda)$ defined by $$M(\lambda) = \Ue/(\Ue \cdot E+\Ue \cdot (K-\lambda)).$$ Equivalently, $M(\lambda)$ has for highest weight vector the coset $v_0$ of $1\in \Ue$, and has the basis $v_0, v_1, v_2,\ldots$ where $v_i$ is the coset of $F^i$ and the action of $\Ue$ is given by
$$Kv_j = \lambda \e^{-2j}v_j\quad ,\quad F v_j = v_{j+1}$$ $$E v_j = \left\lbrace \begin{array}{ll}[j]\frac{\lambda \e^{1-j}-\lambda^{-1}\e^{j-1}}{\e-\e^{-1}} v_{j-1}, & \mbox{if} \ j>0\\ 0,& \mbox{if} \ j=0.\end{array}\right.$$
The linear independence of the $v_i$ follows from Theorem \ref{PBWteo}. From the formulas we see that the modules $V_r^\pm$ and $V(\lambda,0,b)$ are highest weight $\Ue$-modules of highest weights $\pm \e^{r}$ and $\lambda$, respectively, and that 
$$\ V_r^\pm = M(\pm \e^r)/M_r\ ,\ V(\lambda,0,b) = M(\lambda)/\Ue \cdot(v_n-bv_0),$$ 
where $M_r$ is the submodule of $M(\pm \e^r)$ spanned by the $v_i$ with $i>r$. Note that $\Ue \cdot(v_n-bv_0)$ is spanned by all $F^i(v_n-bv_0) = v_{i+n}-bv_i$, since $[n]=0$ implies $Ev_n=Ev_0=0$, and $K(v_n-bv_0) = \lambda(v_n-bv_0)$. 

The modules $V_r^\pm$ and $M(\lambda)$ extend to modules $V_{r,q}^\pm$ and $M(\lambda)_q$ over $\Uq$ for all values of $q$ such that $q^2\ne 1$, by replacing $\e$ by $q$ in all formulas. The $\Uq$-modules $V_{r,q}^\pm$ are defined in any dimension, hence with no bound on $r$, and still satisfy $V_{r,q}^\pm= M(\pm q^r)_q/M_r$. They are simple modules, and any simple $\Uq$-module is isomorphic to some $V_{r,q}^\pm$ (\cite[Prop. 2.6]{J}, \cite[Th. VI.3.5]{K}). In fact, $M(\lambda)_q$ is simple if and only if $\lambda\ne \pm q^l$ for all integers $l\geq 0$. When $\lambda= \pm q^l$, the $v_i$ for $i>l$ span a submodule of $M(\lambda)_q$ isomorphic to $M(q^{-2(l+1)}\lambda)$, and this is the only non trivial proper submodule of $M(\lambda)_q$ \cite[Prop. 2.5]{J}. \smallskip

\subsubsection{Symmetries} We have just claimed that the modules $M(\lambda)_q$ are not simple when $\lambda= \pm q^l$ for some integer $l\geq 0$. This follows from the fact that the vector $v_{l+1}$ satisfies
$$Ev_{l+1} =  \pm [l+1]\frac{q^l q^{1-(l+1)}-q^{-l}q^{(l+1)-1}}{q-q^{-1}} v_{l} = 0$$
and 
$$Kv_{l+1} =  \pm q^lq^{-2(l+1)}v_{l+1} = \pm q^{-l-2}v_{l+1}.$$
Hence we have an injection
\begin{equation}\label{inject}
 M(\pm q^{-l-2})_q \lra M(\pm q^l)_q
\end{equation}
mapping the highest weight vector $v_0$ of $M(\pm q^{-l-2})_q$, which is a simple $U_q$-module since $-l-2\leq 0$, to $v_{l+1}$. 

In particular, by specializing at $q=\e$ and taking the quotient by $\Ue \cdot(v_n-bv_0)$, we get that $V(\pm \e^r,0,0)$ is simple if and only if $r=n-1$, and there is for every $0\leq r\leq n-2$ a (non split) short exact sequence of $\Ue$-modules:
\begin{equation}\label{exactseq}
 0\lra V_{n-r-2}^\pm \lra V(\pm \e^r,0,0)\lra V_r^\pm \lra 0.
\end{equation}
The inclusion $V_{n-r-2}^\pm \ra V(\pm \e^r,0,0)$ maps the highest weight vector $v_0\in V_{n-r-2}^\pm$ to $v_{r+1}\in V(\pm \e^r,0,0)$. 

\begin{xca} \label{sym} \textit{(Application: the image of the Harish-Chandra homomorphism.)}  Recall that $\pi$ maps a central element $z = \sum_{i\geq 0} F^ih_iE^i$ of degree $0$ to the Laurent polynomial $\pi(z)=h_0 \in \mc[K,K^{-1}]$ (see \eqref{HCdef}). 

\noindent (a) Prove a statement analogous to Lemma \ref{scacentre} for finite dimensional highest weight $\Uq$-modules $V$ of highest weight $\lambda\in \mc^*$ (possibly, not simple): show that when $q$ is not a root of unity (resp. when $q=\e$), for all $v\in V$ and $z\in Z_q$ (resp. $Z_\e\cap U_0$) we have
\begin{equation}\label{hwc} zv = \pi(z)(\lambda)\ v.\end{equation}
(b) Deduce from \eqref{inject} and \eqref{hwc} that when $q$ is not a root of unity, for all $z\in Z_q$ and integer $l\geq 0$ we have
$$\pi(z)(\pm q^{l-1}) = \pi(z)(\pm q^{-l-1}).$$  
Similarly, deduce from \eqref{exactseq} and \eqref{hwc} that for all $z\in Z_\e\cap U_0$ and $1\leq r\leq n-1$ we have
$$\pi(z)(\pm \e^{r-1}) = \pi(z)(\pm \e^{-r-1}).$$
(c) Show that the monomials $K^{in}(K+K^{-1})^j$ and $K^l$, where $i\in \mz$, $0\leq j<n$, and $0<l<n$, form a linear basis of $\mc[K,K^{-1}]$. 

\noindent (d) Deduce from (b) and (c) that when $q$ is not a root of unity (resp. when $q=\e$), the homomorphism $\gamma_{-1}\circ \pi$ maps $Z_q$ (resp. $Z_\e\cap U_0$) to 
$\mc[K+K^{-1}]$ (resp. the polynomial subalgebra $\mc[K^n,K^{-n},K+K^{-1}]$ of $\mc[K,K^{-1}]$).
\end{xca}
By the discussion preceding Theorem \ref{amainteo}, we know that  $\gamma_{-1}\circ \pi$ is injective over $Z_q$ (resp. $Z_\e \cap U_0'$) when $q$ is not a root of unity (resp. when $q=\e$), and maps the Casimir element $\Omega$ to a scalar multiple of $K+K^{-1}$. Also, $Z_\e$ is generated by $E^n$, $F^n$ and $Z_\e \cap U_0'$, and $\gamma_{-1}\circ \pi(Z_\e \cap U_0') = \gamma_{-1}\circ \pi(Z_\e \cap U_0)$. Then, by Exercise \ref{sym} we have:
\begin{proposition}\label{generators}\textit{(i) When $q$ is not a root of unity, we have an isomorphism} $$\gamma_{-1}\circ \pi \colon Z_q\stackrel{\cong}{\lra} \mc[K+K^{-1}].$$ 

\noindent \textit{(ii) When $q=\e$, we have an isomorphism $$\gamma_{-1}\circ \pi \colon Z_\e\cap U_0'\stackrel{\cong}{\lra} \mc[K^n,K^{-n},K+K^{-1}]$$
mapping $\mc[\Omega]$ to $\mc[K+K^{-1}]$. Moreover, $Z_\e$ is generated by $E^n$, $F^n$, $K^{\pm n}$ and $\Omega$.}
 \end{proposition}
\subsubsection{Classification} The following result can be proved by elementary methods \cite[\S 2.11-2.13]{J}.
\begin{theorem} \label{class} The $\Ue$-modules $V_r^\pm$, $0\leq r\leq n-1$, and $V(\lambda,a,b)$ are simple, and any non zero simple $\Ue$-module is isomorphic to one of them.
\end{theorem}
\begin{remark}\label{Vnonunique} When $b\ne 0$ the module $V(\lambda,a,b)$ uniquely determines $b$ but not $\lambda$ and $a$. In fact, we get an isomorphic module by replacing $v_0$ by any other $v_i$, that is, $\lambda$ by $\lambda\e^{-2i}$ and $a$ by $a+[i](\lambda\e^{1-i}-\lambda^{-1}\e^{i-1})(\e-\e^{-1})^{-1}b^{-1}$. 
\end{remark}
By Proposition \ref{generators}, any central character $\chi=\chi_V\in {\rm Spec}(Z_\e)$ is determined by its values at the generators $x$, $y$, $z$ and $\Omega$ of $Z_\e$. Put \begin{equation}\label{coordinates} (x_{\chi},y_{\chi},z_{\chi},c_{\chi}) = (\chi(x),\chi(y),\chi(z),\chi(\Omega))\in \mc^4.\end{equation}
We have:
\begin{itemize} \item if $V=V_r^\pm$, then \begin{equation}\label{value0} \xg = 0\ ,\ \yg = 0\ ,\ \zg= \pm 1\ ,\ \cg = \pm \frac{\e^{r+1}+\e^{-1-r}}{(\e-\e^{-1})^2};\end{equation}
 \item If $V=V(\lambda,a,b)$, then 
\begin{equation}\label{value1}
\xg = -(\e-\e^{-1})^n\lambda^{-n}a\prod_{j=1}^{n-1} \left(ab+[j]\frac{\lambda\e^{1-j}-\lambda^{-1}\e^{j-1}}{\e-\e^{-1}}\right) \end{equation}\begin{equation}\label{value2}\yg = (\e-\e^{-1})^nb\ ,\ \zg=\lambda^n\ ,\ \cg = ab + \frac{\lambda\e+\lambda^{-1}\e^{-1}}{(\e-\e^{-1})^2}.\end{equation}
\end{itemize}
Set 
$$c_r^\pm = \pm \frac{\e^{r+1}+\e^{-1-r}}{(\e-\e^{-1})^2}.$$
From \eqref{exactseq} or by a direct computation, we get 
\begin{equation}\label{memecasimir} c_r^{\pm} = c_{n-r-2}^{\pm},\quad 0\leq r\leq n-2.
\end{equation}
Hence there are $n-1$ distinct values $c_r^{\pm}$ of the Casimir element, achieved at $r=0,1,\ldots,(n-3)/2$. 

It will be useful to distinguish among central characters and $\Ue$-modules by using the map \begin{equation}\label{chi}
 \varphi\colon {\rm Rep}(\Ue)\stackrel{\Xi}{\lra} {\rm Spec}(Z_\e)\stackrel{\tau}{\lra} {\rm Spec}(Z_0) \stackrel{\sigma}{\lra} PSL_2\mc^0.
\end{equation} 
Recall the big cell decomposition $PSL_2\mc^0 = U_-TU_+$. Set
$$\mathcal{D} = \{\chi_r^{\pm}\in {\rm Spec}(Z_\e)\ | \ (x_{\chi_r^{\pm}},y_{\chi_r^{\pm}},z_{\chi_r^{\pm}},c_{\chi_r^{\pm}}) = (0,0,\pm 1,c_r^{\pm}),\ 0\leq r\leq (n-3)/2\}.$$
By \eqref{value0}-\eqref{memecasimir} we have \begin{equation}\label{presing} \Xi^{-1}(\chi_r^{\pm}) = \{V_r^{\pm},V_{n-r-2}^{\pm}\}\ ,\ \varphi^{-1}({\rm Id}) = \Dd.\end{equation}
\begin{definition}\label{defclas} A simple $\Ue$-module $V$ and its central characters $\Xi(V)$ and $\tau \circ \Xi(V)$ are called:
\begin{itemize}
\item \emph{diagonal} (resp. \emph{triangular}) if $\varphi(V) \in T$ (resp. $\varphi(V)\in PB_\pm = U_-T$ or $TU_+$);
\item \emph{regular} (resp. \emph{singular}) if $\varphi(V)\ne {\rm Id}$ (resp. $\varphi(V) = {\rm Id}$);
\item \emph{regular semisimple} if the conjugacy class of $\varphi(V)$ intersects $T\setminus {\rm Id}$. 
\item \emph{cyclic} if $E^n$ and $F^n$ act as non zero scalars. 
\end{itemize} 
\end{definition}
So, a point $\chi \in {\rm Spec}(Z_\e)$ is regular if $\chi\notin \mathcal{D}$. Regular semisimple characters are sent by $\pi$ to loxodromic elements of $PSL_2\mc$; the regular diagonal $U_\e$-modules are the $V(\lambda,0,0)$s with $\lambda^n\ne \pm 1$. 

The cyclic simple $\Ue$-modules have the property that any two eigenvectors of $K$ can be obtained one from each other by applying some power of $E$ or $F$; matrix realizations are immediately derived from the formulas of $V(\lambda,a,b)$ when $ab\ne 0$. They are complementary to the highest weight modules in Rep$(\Ue)$, since a module on which $F^n$ acts as zero and $E^n$ does not is isomorphic to some $V(\lambda,0,b)$ by applying the Cartan automorphism \eqref{Cartan}. In particular, the cyclic $U_\e$-modules map under $\varphi$ to a dense subset of $PSL_2\mc$, and together with the diagonal modules they cover a neighborhood of the identity. 
\smallskip

\subsection{Quantum coadjoint action and the bundle $\Xi_M$} We are going to show that the central character map $\Xi\colon {\rm Rep}(U_\e)\ra {\rm Spec}(Z_\e)$ can be used to define a bundle $\Xi_M$ of $\Ue$-modules over ${\rm Spec}(Z_\e)$, endowed with an action of the subgroup $\Gg$ of Aut$(\hat{U}_\e)$.\smallskip

Let $V\in {\rm Rep}(U_\e)$, and $$\rho_V:U_\e\ra {\rm End}(V)$$ be the corresponding linear representation of $\Ue$. Recall \eqref{extalg} and Definition \eqref{qcagroup}. Since $Z_0$ acts by scalar operators on $V$, the same is true for $\hat{Z}_0$, which consists of holomorphic functions of $x$, $y, z^{\pm 1}\in Z_0$. Hence $V$ is naturally a $\hat{U}_{\e}$-module. Since any $\hat{U}_{\e}$-module is uniquely determined by the action of $\Ue$ on it, we will henceforth identify $\Ue$-modules and $\hat{U}_{\e}$-modules. 
\begin{definition} Given $g\in \Gg$, the \emph{twisted} $\Ue$-module $^gV$ is defined by 
 $$\rho_{^g\!V}(u) = \rho(gu),\quad u\in \Ue.$$
\end{definition}
Note that $\chi_{^g\!V} = \chi_V\circ g$ for all $g\in \Gg$ and $\chi\in {\rm Spec}(Z_\e)$, so the action of $\Gg$ on simple $\Ue$-modules lifts the opposite of the action of $\Gg$ on Spec$(Z_\e)$. \smallskip

The basic properties of the central character map $\Xi$ are given by the following result. It is a special case of \cite[\S 3.7-3.8 $\&$ Th. 4.2]{DCK} (see also \cite[\S 20]{DCP}), which applies to the quantum groups of an arbitrary complex semi-simple Lie algebra. Recall the algebras $\Ue^{\g}$ in \eqref{redalg}.
\begin{theorem}\label{moduleteo1} (i) If $\chi\in {\rm Spec}(Z_\e)\setminus \mathcal{D}$, then $\Ue^{\g} \cong M_n(\mc)$. Hence there is up to isomorphism a unique simple $U_\e$-module $V_{\chi}$ with central character $\chi$; we have $V_{\chi} \cong V(\lambda,a,b)$ for any $(\lambda,a,b)\in \mc^*\times \mc^2$ satisfying \eqref{value1}-\eqref{value2}.

(ii) If $\chi\in \mathcal{D}$, there are exactly two simple $U_\e$-modules $V_{\g}$ with central character $\chi$: if $(x_{\chi},y_{\chi},z_{\chi},c_{\chi}) = (0,0,\pm 1,c_r^\pm)$ we have $\Xi^{-1}(\g) = \{V_r^\pm,V_{n-r-2}^\pm\}$. 

(iii) The restriction map $\tau\colon {\rm Spec}(Z_\e)\ra {\rm Spec}(Z_0)$ has degree $n$, and the coordinate ring $Z_\e$ is generated by $Z_0$ and $\Omega$ subject to any one of the equivalent relations $(R_{\pm})$ given by
\begin{equation}\label{ext}
 \prod_{j=0}^{n-1}\left( \Omega - c_j^{\pm}\right) = E^nF^n+\frac{K^n+K^{-n}\mp 2}{(\e-\e^{-1})^{2n}}.
\end{equation}  
where, as usual, $c_j^{\pm} = \pm \frac{\e^{j+1}+\e^{-1-j}}{(\e-\e^{-1})^2}$.
\end{theorem}
\begin{proof} (i) Consider a cyclic central character $\chi$ (so $\chi\in {\rm Spec}(Z_\e)\setminus \mathcal{D}$). For any $V\in \Xi^{-1}(\chi)$, $\rho_V$ induces an irreducible representation $$\bar{\rho}_V:\Ue^{\chi} \ra {\rm End}(V).$$ Because $\chi$ is cyclic, dim$(V)=n$, and $\Ue^{\chi}$ has no non trivial proper ideal (by Theorem \ref{PBWteo} the unit necessarily belongs to any non zero ideal). Hence $\bar{\rho}_V$ is faithful. Then, Wedderburn's theorem implies that $\bar{\rho}_V$ is an isomorphism, that is, $\Ue^{\chi} \cong M_n(\mc)$ \cite[\S XVII.3]{L}. Conjugation with $\Gg$ yields a similar isomorphism for any module in the $\Gg$-orbit of $V$. Since any non trivial conjugacy class of $PSL_2\mc$ contains an element whose entries are all non zero, by Theorem \ref{QCAteo} the $\Gg$-orbit of any character $\chi\in {\rm Spec}(Z_\e)\setminus \Dd$ contains one which is cyclic. Hence $\Ue^{\chi} \cong M_n(\mc)$. Since $M_n(\mc)$ is a simple ring \cite[\S XVII.4 $\&$ 5]{L}, this implies the uniqueness of $V_{\chi}\in \Xi^{-1}(\chi)$ when $\chi\notin \Dd$.

(ii) is the content of \eqref{presing}. 

(iii) Because of (i), in order to prove that ${\rm deg}(\tau) = n$ it is enough to find an open subset $\Oo$ of Spec$(Z_0)$ (in the complex topology) such that for all $\tau(\chi)\in \Oo$, $(\tau\circ\Xi)^{-1}(\tau(\chi))$ consists of $n$ simple $\Ue$-modules. In fact, this is true for any regular diagonal $Z_0$-character $\tau(\chi)$, for $(\tau\circ\Xi)^{-1}(\tau(\chi))$ consists of the $V(\lambda,0,0)$ with $\lambda^n = z_{\chi}$ $(\ne \pm 1)$. Since the $\Gg$-action on Rep$(\Ue)$ preserves the isomorphism type of the fibers of $\Xi$, and in particular their dimension and number of components, it is also true for any $Z_0$-character in the $\Gg$-orbit of a regular diagonal one. Hence we can take $\Oo$ to be the set of regular semisimple $Z_0$-characters. Theorem \ref{QCAteo} implies that $\Oo$ is Zariski open and dense.

By Proposition \ref{generators} we know that $Z_\e$ is generated by $E^n$, $F^n$, $K^{\pm n}$ and $\Omega$. On another hand, \eqref{comEFutile} gives by an easy induction on $r$ the relation 
\begin{equation}\label{Casimirrel} \prod_{j=0}^{r-1}\left( \Omega - \frac{\e^{2j+1}K+\e^{-2j-1}K^{-1}}{(\e-\e^{-1})^{2}}\right)=F^rE^r.
\end{equation}
When $r=n$ this relation makes sense in $Z_\e$. To see this, let us expand the left-hand side as $\sum_{k=0}^{n} (-1)^{k}\sigma_k \Omega^{n-k}$. The coefficients $\sigma_k$ are the elementary symmetric functions of the variables $x_j =\frac{\e^{2j+1}K+\e^{-2j-1}K^{-1}}{(\e-\e^{-1})^{2}}$, $0\leq j\leq n-1$. It is classical that $\sigma_k$ can be expressed as a polynomial with rational coefficient of the power sum functions $t_i = \sum_{j=0}^{n-1} x_j^i$, $0\leq i\leq k$. Then, by using $\sum_{i=0}^{n-1} \e^{2i}=0$ it is immediate to check that all $\sigma_k$ are complex numbers not involving $K$, except $\sigma_{n} = -\frac{K^n+K^{-n}}{(\e-\e^{-1})^{2n}}$. Hence, for $r=n$ we can rewrite \eqref{Casimirrel} as 
$$\prod_{j=0}^{n-1}( \Omega - \alpha_j) = F^nE^n + \frac{K^n+K^{-n}}{(\e-\e^{-1})^{2n}}. $$ 
for some $\alpha_j\in \mc$. The relation \eqref{ext} follows from this by noting that the left-hand side (resp. right-hand side) acts as $0$ (resp. $\pm 2(\e-\e^{-1})^{-2n}$) on all the simple $\Ue$-modules $V_r^{\pm}$ if and only if $\alpha_j = c_j^{\pm}$ for all $0\leq j\leq n-1$. 
\end{proof}
\begin{remark} (i) We have seen in Theorem \ref{finite} that ${\rm rk}_{Z_0}\Ue = n^3$. From this one deduces that ${\rm deg}(\tau)= n$ and $V\in \Xi^{-1}(\chi)$ is unique for generic $\chi$, by using that ${\rm dim}(V)=n$ for regular diagonal characters $\chi$, as in the first part of the proof of Theorem \ref{moduleteo1} (iii).  

\noindent (ii) It follows directly from \eqref{value1}-\eqref{value2} and Remark \ref{Vnonunique} that $\chi\in {\rm Spec}(Z_\e)\setminus \Dd$ determines uniquely $V(\lambda,a,b)$ up to isomorphism. This implies the uniqueness of $V_{\g} \in \Xi^{-1}(\chi)$, but depends on the classification of the simple $\Ue$-modules, which is provided by Theorem \ref{class}. 
 \end{remark}
Following \cite[\S 6]{DCKP} and \cite[\S 11 $\&$ 21]{DCP}, we use Theorem \ref{moduleteo1} to provide the collection of simple regular $\Ue$-modules a structure of vector bundle over the smooth part of Spec$(Z_\e)$. 

To simplify notations, let  
$$G=PSL_2\mc.$$
It will be useful to distinguish between the Cartan subgroup $T=\mc^*$ of $H$, and its image $\bar{T} = T/(\pm 1)$ under the $2$-fold covering $\sigma\colon H\ra G^0$.

Recall that $G$ is an \textit{affine} algebraic group, since the adjoint representation ${\rm Ad}\colon G\ra {\rm Aut}(\mathfrak{g})$ identifies $G$ with a subgroup of the complex orthogonal group $SO_3\mc$ for the Killing form of $\mathfrak{g}$. Denote by $G/\!/G$ the affine variety with coordinate ring $\mc[G]^G$, the ring of regular functions on $G$ invariant under conjugation. The points of $G/\!/G$ are in one to one correspondence with the conjugacy classes of elements of $G$ having distinct traces up to sign, and we have isomorphisms
\begin{equation}\label{git1} \mc[G]^G \cong \mc[\bar{T}]^W \ ,\ G/\!/G \cong \mc^*/(z\sim \pm z^{-1}).
\end{equation}
Here $W \cong \mz/2$ is the Weyl group of $G$ acting on the torus $\bar{T}\subset G$ by inversion, and $z$ is the coordinate function of $T\subset H$. Consider the maps
\begin{equation}\label{fp}
 p_1= p\circ\sigma\colon H\lra G/\!/G\ ,\ p_2 : G/\!/G \lra G/\!/G,
\end{equation}
where $p\colon G\ra G/\!/G$ is the quotient map, and $p_2$ is induced by the $n$-th power map $g\mapsto g^n$, $g\in G$. The  \textit{fibered product} of $p_1$ and $p_2$ is the variety 
\begin{equation}\label{fibprod}  H \times_{G/\!/G} G/\!/G = \{ (h,\bar{g})\in H\times G/\!/G\ |\ p_1(h) = p_2(\bar{g})\}.
\end{equation}
\begin{theorem} \label{bundle} (i) The action of $\Gg$ on Spec$(Z_0)$ extend to Spec$(Z_\e)$ by a trivial action on Spec$(\mc[\Omega])$. Hence, if $\mathcal{O}$ is an orbit of $\Gg$ in Spec$(Z_0)$, the connected components of $\tau^{-1}\mathcal{O}$ are orbits of $\Gg$ in Spec$(Z_\e)$.

\noindent (ii) {\rm Spec}$(Z_\e)$ is a $3$-dimensional affine algebraic variety with singular set $\Dd$, and is isomorphic to $H \times_{G/\!/G} G/\!/G$.
\end{theorem}
\begin{proof} (i) The claim follows from Proposition \ref{generators} (ii) and the fact that for all $q$ we have $\Omega\in Z_q$, so that $D_a(\Omega)=0$ for all $a\in Z_0$.  

\noindent (ii) According to Theorem \ref{moduleteo1} (iii), {\rm Spec}$(Z_\e)$ is the hypersurface in $\mc^4$ given in coordinates $(e,y,z,c)$ by (see \eqref{pref}-\eqref{prefbis})
\begin{equation}\label{ext2}
 (\e-\e^{-1})^{2n}\prod_{j=0}^{n-1}\left( c - c_j^{\pm}\right) = ey+ z+z^{-1}\mp 2.
\end{equation} It follows from \eqref{memecasimir} that the singularities of {\rm Spec}$(Z_\e)$ are quadratic and coincide with $\Dd$ (note, in particular, that $(0,0,\pm 1,c_{n-1}^{\pm})$ is a smooth point).

Let us prove now the isomorphism with \eqref{fibprod}. By Proposition \ref{generators} (ii) we have
$$Z_\e =  Z_0\otimes_{Z_0\cap \mc[\Omega]} \mc[\Omega],$$
and $\mc[\Omega]\cong \mc[K+K^{-1}]$. Let us identify $\mc[K^{\pm 1}]$ with $\mc[T]$, the coordinate ring of the torus $T$ of $H$. Then $\mc[\Omega]\cong \mc[T]^W$. By Proposition \ref{center}, $Z_0 \cong \mc[H]$. The Harish-Chandra homomorphism applied to \eqref{ext} gives $Z_0 \cap \mc[\Omega]  \cong \mc[K^n+K^{-n}]$. On another hand, $\mc[K^n+K^{-n}]\cong \mc[T/\mu_n]^W$, where $\mu_n$ is the subgroup of $T$ of $n$-th roots of unity. (Alternatively, Theorem \ref{QCAteo} gives $Z_0^{\Gg} = \mc[T/\mu_n]^W$, where $\mc[T/\mu_n]^W$ is identified with the subalgebra of $Z_0$ generated by $z+z^{-1}$, and so $Z_0 \cap \mc[\Omega]  = \mc[T/\mu_n]^W$ by (i) above). Hence
$$Z_\e\cong \mc[H] \otimes_{\mc[T/\mu_n]^W} \mc[T]^W.$$
Since $n$ is odd, this is equivalent to
$$Z_\e\cong  \mc[H] \otimes_{\mc[\bar{T}/\mu_n]^W} \mc[\bar{T}]^W.$$
The isomorphism of Spec$(Z_\e)$ with $H \times_{G/\!/G} G/\!/G$ follows by duality.
\end{proof} 
Denote by $X = {\rm Spec}(Z_\e)\setminus \Dd$ the subset of regular central characters. By Theorem \ref{bundle}, it coincides with the smooth part of Spec$(Z_\e)$. 

Take a point $\chi\in X$. Any choice of parameters $(\lambda,a,b)$ as in \eqref{value1}-\eqref{value2} determines an isomorphism of $\Ue^{\chi}$-modules $V_\chi \cong V(\lambda,a,b)$, and an isomorphism of algebras $\Ue^{\chi} \cong M_n(\mc)$ by identifying $V(\lambda,a,b)$ with $\mc^n$. Both isomorphisms are varying smoothly with respect to $(\lambda,a,b)$ over a sufficiently small open neighborhood $O_{\chi}$ of $\chi$ in $X$. Hence we have coordinate charts $\textstyle g_{O_{\chi}}\colon \coprod_{\rho\in O_\chi} \Ue^{\rho} \ra O_{\chi} \times M_n(\mc)$ with smooth transition functions $$g_{O_{\chi}}\circ g_{O_{\chi'}}^{-1}\colon (O_{\chi}\cap O_{\chi'}) \times M_n(\mc) \ra (O_{\chi}\cap O_{\chi'}) \times M_n(\mc)$$ that restrict to algebra automorphisms on the second component. By taking a locally finite covering of $X$ made of neighborhoods $O_\chi$ we thus obtain a trivializing atlas for a smooth vector bundle of matrix algebras 
\begin{equation}\label{matrixbundle} \Xi_A \colon A_\e \ra X,
\end{equation} 
where $\textstyle A_\e = \coprod_{\chi \in X} \Ue^{\chi}$ as a set. By a ``smooth bundle of algebras with unit'' we mean that a smooth product with unit is defined on sections. For such bundles the functions on the base can be identified with multiples of the unit section. 

Similarly, we have a rank $n$ vector bundle 
$$\Xi_M \colon M_\e \ra X$$ 
with fibers the $\Ue$-modules $\Xi_M^{-1}(\chi) = V_\chi$. Clearly, $\Xi_A$ and $\Xi_M$ are associated bundles via the action of $\Ue^{\chi}$ on $V_\chi$. They are topologically non trivial, since by using Remark \ref{Vnonunique} we can find a loop in Spec$(Z_\e)\setminus \Dd$ with non trivial holonomy. Namely, put $\e := e^{\frac{2\sqrt{-1}\pi d}{n}}$. For $b\ne 0$, the path in $\mc^*\times \mc^2$ given by
$$\gamma\colon t\longmapsto \left(\lambda e^{-\frac{4\sqrt{-1}\pi d t}{n}},a+t\frac{\lambda-\lambda^{-1}}{b(\e-\e^{-1})},b\right),\quad t\in [0,1],$$ 
projects to a loop in Spec$(Z_\e)$ with holonomy in $\Xi_M$ the permutation matrix $(\delta_{i+1,j})$ (indices mod $n$).
\begin{theorem}\label{bundle2} The group $\Gg \subset {\rm Aut}(\hat{U}_\e)$ acts on $\Xi_A$ and $\Xi_M$ by bundle morphisms. 
\end{theorem}
In particular, the orbits of $\Gg$ correspond to symplectic leaves (resp. conjugacy classes) of Spec$(Z_\e)\setminus \Dd$ (resp. $PSL_2\mc^0$).
\smallskip

\begin{proof} Since by Proposition \ref{qca} the group $\Gg$ maps the algebra of functions on Spec$(Z_\e)$ into itself, its acts linearly on the fibers of $\Xi_A$ and $\Xi_M$ by automorphisms of algebras and $\Ue$-modules, respectively. More precisely, for any $g\in \Gg$ and $\chi\in {\rm Spec}(Z_\e)$, the action of $g^{-1}$ on $\hat{U}_\e$ maps isomorphically the ideal $\mathcal{I}^{\chi}$ to $\mathcal{I}^{g.\chi}$, and hence the algebra $\Ue^{\chi}$ to $\Ue^{g.\chi}$ and its simple module $V_\chi$ to $V_{g.\chi}$.\end{proof}

\begin{remark} (a) In the terminology of \cite{RVW}, Theorem \ref{bundle2} reflects the fact that the pair $(\Ue,Z_0)$ is a \textit{Poisson fibered algebra}. The infinitesimal action of $\Gg$ on $\Xi_A$ defines a morphism of vector bundles $D\colon \Omega^1X\ra \Aa_{\e}$ with non trivial curvature, where $\Aa_{\e}$ is the bundle of first order differential operators on $\Xi_A$ with symbols ${\rm Id}\otimes \xi$, $\xi \in TX$. 

\noindent (b) The bundle $\Xi_M$ extends to the whole of Spec$(Z_\e)$, with singular (non simple) fibers $V(\pm \e^r,0,0)$, $0\leq r\leq (n-3)/2$, over $\Dd$.
\end{remark}
\section{Intertwinners: the category $\Ue$-Mod}\label{CGO}
The Hopf algebra structure of $\Ue$ endows the category $\Ue$-Mod of finite dimensional left $\Ue$-modules with a tensor product and a duality. We are going to describe them in geometric terms when applied to \textit{regular} $\Ue$-modules, by using the theorems \ref{bundle} and \ref{bundle2}. To clarify the picture, after some preliminaries  we recall well-known results on a subcategory based on singular modules and generating the Reshetikhin-Turaev TQFT.
\subsection{A few basic definitions \cite{K}} Unless stated otherwise all the modules we are going to consider will be left modules, finite dimensional over $\mc$.

The tensor product vector space $V\otimes W$ of two $\Ue$-modules is naturally a $\Ue \otimes \Ue$-module. It is made into a $\Ue$-module by setting
\begin{equation}\label{tensprod} a.(v\otimes w) = \Delta(a).(v\otimes w)\end{equation} for all $a\in\Ue$, $v\in V$, $w\in W$, where $\Delta(a)\in \Ue \otimes \Ue$ is the coproduct of $a$. This action is compatible with the natural tensor product of linear maps of $\Ue$-modules, so it defines a bifunctor $\otimes\colon$ $\Ue$-Mod $\times$ $\Ue$-Mod $\rightarrow$ $\Ue$-Mod.

Consider the trivial action of $\Ue$ on $\mc$ given by the counit, $a.z = \eta(a)z$. By \eqref{coass}, for all $\Ue$-modules $U,V,W$ the formula \eqref{tensprod} turns 
the canonical isomorphisms of vector spaces $l_V\colon \mc\otimes V \cong V$, $r_V\colon V \otimes \mc\cong V$ and \begin{equation}\label{associator} a_{U,V,W}\colon (U\otimes V)\otimes W \stackrel{\cong}{\lra} U\otimes (V\otimes W)
\end{equation} 
into isomorphisms of $\Ue$-modules which are natural with respect to morphisms of $\Ue$-modules. Hence they define \textit{natural isomorphisms} of functors $l$,$r$ and $a$, with obvious compatibility relations. In particular    
\begin{equation}\label{nat-trans} a\colon \otimes (\otimes \times {\rm id})\longrightarrow \otimes ({\rm id}\times \otimes)\end{equation}
makes the following \textit{Pentagonal Diagram} commutative: 
\begin{equation}\label{pentdiag} \xymatrix{(U\otimes (V\otimes W))\otimes X \ar[dd]^{a_{U,V\otimes W,X}} & & ((U\otimes V)\otimes W)\otimes X \ar[ll]_{a_{U,V,W}\otimes {\rm id}_X} \ar[d]^{a_{U\otimes V,W,X}} \\ & & (U\otimes V)\otimes (W\otimes X) \ar[d]^{a_{U,V,W\otimes X}}\\ U\otimes ((V\otimes W)\otimes X) \ar[rr]^{{\rm id}_U\otimes a_{V,W,X}} & & U\otimes (V\otimes (W\otimes X)).}\end{equation}
The category $\Ue$-Mod endowed with $(\otimes,a,l,r)$ is an exemple of \textit{tensor category}, with \textit{associativity constraint} $a$ and \textit{unit} $\mc$. Note that $a$ explicits the different module structures at both sides.

The category $\Ue$-Mod has also a \textit{(left) duality}, that is a pair $(b,d)$ of natural transformations given on any $\Ue$-module $V$ by morphisms
\begin{equation}\label{duality} \begin{array}{cccc} d_V\colon & V^*\otimes V& \longrightarrow & \mc\\ b_V\colon & \mc & \rightarrow & V\otimes V^*,\end{array}
\end{equation}
where $V^*$ is an $\Ue$-module to be specified, satisfying 
$$({\rm id}_V \otimes d_V)(b_V \otimes {\rm id}_V) = (d_V \otimes {\rm id}_{V^*}) ({\rm id}_{V*} \otimes b_V) = {\rm id}_{V^*}.$$
Naturally we put $V^*={\rm Hom}_{\mc}(V,\mc)$, the dual linear space, and define $d_V$ and $b_V$ as the canonical pairing of vector spaces between $V$ and $V^*$, and the map taking $1\in \mc$ to $\sum_i v_i\otimes v^i$, where $\{v_i\}$ and $\{v^i\}$ are dual basis of $V$ and $V^*$. Then a (\textit{left}) $\Ue$-module structure is defined on $V^*$ by
\begin{equation}\label{dualmod}
d_V((a.\xi)\otimes v) = d_V(\xi\otimes (S(a).v))\end{equation}
for all $a\in\Ue$, $v\in V$ and $\xi\in V^*$ (recall that the antipode $S\colon \Ue \ra \Ue$ is an \textit{anti}automorphism). By using \eqref{antipode} one can check that $d_V$ and $b_V$ are $\Ue$-linear maps. Any $\Ue$-linear map $f\colon V\rightarrow W$ has a \textit{transpose} $f^*\colon W^*\rightarrow V^*$ given by
$$f^* = (d_W\otimes {\rm id}_{V^*})({\rm id}_{W^*}\otimes f\otimes {\rm id}_{V^*})({\rm id}_{W^*}\otimes b_V),$$ so that $^*$ defines a contravariant functor $\Ue$-Mod $\rightarrow$ $\Ue$-Mod. As in the case of linear spaces there are natural bijections ${\rm Hom}_{\Ue}(U\otimes V,W) \cong {\rm Hom}_{\Ue}(U,W\otimes V^*)$ and ${\rm Hom}_{\Ue}(U^*\otimes V,W) \cong {\rm Hom}_{\Ue}(V,U\otimes W)$. Moreover, we get from \eqref{copS} an isomorphism of $\Ue$-modules \begin{equation}\label{tensdual} V^*\otimes W^*\cong (W\otimes V)^*.\end{equation}
More generally, the vector space ${\rm Hom}_{\mc}(V,W)$ is an $\Ue$-module by setting
\begin{equation}\label{Hom-mod} a.f(v) = \sum_i a_i'.f(S(a_i'').v)
\end{equation}
for all $f\in {\rm Hom}_{\mc}(V,W)$ and $v\in V$, where we put $\Delta(a) = \sum_i a_i'\otimes a_i''$. Note that by counitality in \eqref{coass}, the action \eqref{Hom-mod} reduces to \eqref{dualmod} when $W=\mc$, and \eqref{tensprod} and \eqref{dualmod} imply that the canonical $\mc$-linear isomorphism
\begin{equation}\label{isoHom} \fonc{\lambda_{V,W}}{W\otimes V^*}{{\rm Hom}_{\mc}(V,W)}{w\otimes \xi}{\left(v\mapsto \xi(v) w\right)}
\end{equation}
is an isomorphism of $\Ue$-modules. We have $\lambda_{V,W}^{-1}(f) =(f\otimes {\rm id}_{V^*})b_V$. 
\smallskip

Duality allows one to define a trace in the category. For that, we use the remarkable fact that the square of the antipode $S$ is an inner automorphism: for all $u\in \Ue$ we have $S^2(u) = KuK^{-1}$.  Then, the linear isomorphism
$$\fonc{\Phi_V}{V}{V^{**}}{v}{d_V(\ \cdot \otimes K.v)}$$
is also $\Ue$-linear. The \textit{quantum trace} of $V$ is the $\Ue$-linear map
$$\xymatrix{{\rm tr}_q\colon {\rm End}_{\mc}(V)\ar[r]^{\ \quad \lambda_{V,V}^{-1}} & V\otimes V^* \ar[rr]^{\Phi_V\otimes {\rm id}_{V^*}} & & V^{**}\otimes V^*\ar[r]^{\ \quad d_{V^*}} & \mc.}$$
Explicitly, for all $f\in {\rm End}_{\mc}(V)$ we have
\begin{equation}\label{qtr}
 {\rm tr}_q(f) = {\rm tr}\left(v\mapsto K.f(v)\right),
\end{equation}
where tr is the usual trace map of linear spaces. The \textit{quantum dimension} of $V$ is ${\rm dim}_q(V) = {\rm tr}_q({\rm id}_V)$. We say that $V$ has \textit{trace zero} if for all $\Ue$-linear endomorphisms $f$ of $V$ we have
$${\rm tr}_q(f)=0.$$
Since $\Delta(K) = K\otimes K$ ($K$ is said to be \textit{group like}), the quantum trace is multiplicative: for all $f\in {\rm End}_{\mc}(V)$, $g\in {\rm End}_{\mc}(W)$ the quantum trace of $f\otimes g\in {\rm End}_{\mc}(V\otimes W)$ is \begin{equation}\label{qtrmult} {\rm tr}_q(f\otimes g) = {\rm tr}_q(f){\rm tr}_q(g).\end{equation}
\begin{xca} (a) Check that all the $\Ue$-linear maps above are indeed $\Ue$-linear, and that the isomorphism $\lambda_{(V,W)}: W^*\otimes V^*\ra (V\otimes W)^*$ in \eqref{tensdual} can be decomposed as
$$\lambda_{(V,W)} = (d_W\otimes {\rm id}_{(V\otimes W)^*})({\rm id}_{W^*}\otimes d_V\otimes {\rm id}_{W}\otimes {\rm id}_{(V\otimes W)^*})({\rm id}_{W^*}\otimes {\rm id}_{V^*}\otimes b_{V\otimes W}).$$

(b) {\it (Right duality)} Show that we have $\Ue$-linear maps $d'_V\colon V\otimes ^*\!\! V\rightarrow \mc$ and $b'_V\colon \mc \rightarrow ^*\!\!\! V\otimes V$ analogous to \eqref{duality} but with tensorands permuted, where $^*V$ is defined by \eqref{dualmod} with $d_V$ and $S$ replaced by $d'_V$ and $S^{-1}$. Show that $\xi \mapsto K^{-1}.\xi$ defines an isomorphism of $\Ue$-modules $V^*\longrightarrow ^*\!\! V$, and that ${\rm tr}_q(f) = d_V'(f\otimes {\rm id}_{V^*})b_V$.   
\end{xca}

\subsection{The modular category $\bar{U}_q$-Mod} \label{singcat} The \textit{singular} simple $\Ue$-modules $V_r^{\pm}$ generate a remarkable subcategory of $\Ue$-Mod which has been described in \cite{RT} (in the simply-connected version of $\Ue$, see Remark \ref{simply-connectedUq}).
\smallskip

For all $0\leq r\leq n-1$, the action of $\Ue$ on $V_r^\pm$ factors through the \textit{restricted quantum group}  
\begin{equation}\label{Ures}\bar{U}_\e = \Ue/(E^n=F^n=0, K^{2n}=1).
\end{equation}
The algebra $\bar{U}_\e$ is \textit{finite dimensional}, \textit{not semisimple}, and has for simple modules the \textit{finite set} $\{V_r^\pm, r=0,\ldots,n-1\}$. Note that $V_0^{+} = \mc$ (the trivial module), and \begin{equation}\label{qdimsing}\mbox{if} \ r\leq n-2, \quad {\rm dim}_q(V_r^{\pm})= \pm [r+1]\ne 0.\end{equation}
On another hand, ${\rm dim}_q(V_{n-1}^{\pm})= 0$. Hence, by Schur's lemma, ${\rm tr}_q(f) = 0$ for all $\Ue$-linear endomorphisms $f$ of $V_{n-1}^{\pm}$, which has thus a special status among singular modules, reminiscent of the fact that $V_{n-1}^{\pm} = V(\pm \e^{-1},0,0)$.
\begin{definition} A color is an integer $r$ such that $0\leq r \leq n-2$.
\end{definition}
 There is a unique Hopf algebra structure on $\bar{U}_\e$ such that the quotient map $\Ue \ra \bar{U}_\e$ is a morphism of Hopf algebras (see eg. \cite[Prop. IX.6.1]{K}). Since $V_r^{-} \cong V_0^-\otimes V_r^+$, we can concentrate on the modules $V_r := V_r^+$. They are \textit{self dual}, and satisfy:  

\begin{theorem} \label{singCG} \cite[Th. 8.4.3]{RT} For any colors $i,j$ there is a unique trace zero submodule $Z_{i,j}$ of $V_i\otimes V_j$ such that
\begin{equation}\label{formCG} V_i\otimes V_j \cong \left(\oplus_{k} V_k \right)\oplus Z_{i,j},
\end{equation}
where the sum is over all colors $k$ such that the triple $(i,j,k)$ is $\e$-\emph{admissible}, that is:
\begin{itemize}
\item $i+j+k \in 2\mz$ and $i+j+k\leq 2(n-2)$;
\item $|i-j|\leq k\leq i+j$ (the \emph{triangle inequalities}).
\end{itemize}
When $i+j\leq n-2$ we have $Z_{i,j}=\emptyset$.
\end{theorem}
The modules $Z_{i,j}$ are built on the highest weight modules $V(\e^i,0,0)$ and some $2n$-dimensional extensions thereof. Since their $Z_0$-characters are trivial, like the color modules $V_r$ we can consider them as singular $\Ue$-modules. Their duals and tensor products with color modules $V_r$ decompose into summands of the same form, so the $V_r$ generate a subcategory of $\Ue$-Mod which is closed under tensor product and duality.  
\begin{remark} The category $\bar{U}_\e$-Mod gives rise to the so called \textit{fusion rules} of Wess-Zumino-Witten conformal field theories with gauge group $SU(2)$. In this context, the admissibility conditions of Theorem \ref{singCG} give a method for counting dimensions of the space of conformal blocks \cite[Prop. 1.19]{Ko} (compare also \eqref{6j1} below and \cite[Lemma 2.6]{Ko}).
\end{remark}
The associativity constraint \eqref{associator}, when applied to color modules and computed \textit{modulo trace zero $\bar{U}_\e$-modules}, defines the celebrated \textit{$\e$-$6j$-symbols}. Let us recall how this goes (see \cite[$\S$4]{CFS} for details). 

Theorem \ref{singCG} implies that for each $\e$-admissible triple $(i,j,k)$ the space of $\Ue$-linear embeddings $V_k \ra (V_i\otimes V_j)/Z_{i,j}$ has dimension one. There is a natural basis of lifts, the \textit{Clebsch-Gordan operator}  
\begin{equation}\label{CGO1}
 Y_{i,j}^k\colon V_k\lra V_i\otimes V_j,
\end{equation}
defined in terms of the \textit{Jones-Wentzl idempotents} $e_l$ ($l=i,j,k$) of the Temperley-Lieb algebra; these can be realized as projectors\begin{equation}\label{JW} e_l\colon V_1^{\otimes l}\ra V_1^{\otimes l}\end{equation} in the algebra of $\Ue$-linear transformations of $V_1^{\otimes l}$, whose image is isomorphic to $V_l$ \cite[Prop. 4.3.8]{CFS}. 

By using \eqref{associator} and \eqref{qtrmult}, we get for all colors $a,b,c$ a $\Ue$-linear isomorphism relating two splittings
\begin{align}\label{splitUbar} (V_a\otimes V_b)\otimes V_c = & \oplus_{(l,k)} (Y^l_{a,b}\otimes {\rm id}_{V_c})Y^k_{l,c}(V_k) \oplus Z_1 \notag\\ V_a\otimes (V_b\otimes V_c) = & \oplus_{(j,k)} ({\rm id}_{V_a}\otimes Y^j_{b,c})Y^k_{a,j}(V_k) \oplus Z_2 \end{align}
where the sums are over all pairs of colors $(l,k)$ (resp. $(j,k)$), such that $(a,b,l)$ and $(l,c,k)$ (resp. $(b,c,j)$ and $(a,j,k)$) are $\e$-admissible, and $Z_1$, $Z_2$ are maximal trace zero submodules of $V_a\otimes V_b\otimes V_c$. It follows from Exercise \ref{xcaUbar} (c) below that any $\Ue$-linear map $V_k\ra V_a\otimes V_b\otimes V_c$ can be written \textit{uniquely} as a linear combination of the maps $(Y^l_{a,b}\otimes {\rm id}_{V_c})Y^k_{l,c}$, plus a map whose image is contained in a trace zero summand of $V_a\otimes V_b\otimes V_c$. The \textit{$\e$-$6j$-symbols} 
\begin{equation}\label{6jsymb0}\left\lbrace \begin{array}{ccc} a & b & l \\ c & k & j\end{array}\right\rbrace_\e\in \mc
\end{equation}
are thus defined by
\begin{equation}\label{6j1}({\rm id}_{V_a}\otimes Y^j_{b,c})Y^k_{a,j} = \sum_l \left\lbrace \begin{array}{ccc} a & b & l \\ c & k & j\end{array}\right\rbrace_\e (Y^l_{a,b}\otimes {\rm id}_{V_c})Y^k_{l,c} + S
\end{equation}
where $S$ maps into a summand of trace zero and the sum is taken over all colors $l$ such that $(a,b,l)$ and $(l,c,k)$ are $\e$-admissible. Consider the \textit{normalized} $\e$-$6j$-symbols
\begin{equation}\label{norm6j} \psixj{a}{b}{f}{e}{d}{c} = \frac{(-1)^{f}}{[f+1]}\sqrt{\frac{\Theta(a,b,f)\Theta(d,e,f)}{\Theta(a,c,d)\Theta(b,c,e)}}\left\lbrace \begin{array}{ccc} a & b & f \\ e & d & c\end{array}\right\rbrace_\e 
\end{equation}
where we fix once and for all a square root, and (see \eqref{qbinomialform})
$$\Theta(a,b,k) = (-1)^{\frac{a+b+k}{2}}\frac{[\frac{a+b-k}{2}]![\frac{a-b+k}{2}]![\frac{-a+b+k}{2}][\frac{a+b+k}{2}+1]!}{[a]![b]![k]!}.$$
To an abstract tetrahedron with edges colored by $a,b,c,d,e,f$, let us associate the scalar \eqref{norm6j}. We have: 
\begin{theorem}\label{6jmain} {\rm (\cite{KR}, \cite{RT}; see \cite[Th. 4.4.6]{CFS})} The normalized $\e$-$6j$-symbols \eqref{norm6j} are well-defined, and satisfy:
\begin{itemize}
 \item \textit{Invariance under full tetrahedral symmetries}.
\item The \textit{Elliot-Biedenharn identity:}
\begin{multline}\psixj{c}{d}{h}{g}{e}{f}.\psixj{b}{h}{k}{g}{a}{e} =  \\ \sum_j (-1)^{j}[j+1] \psixj{b}{c}{j}{f}{a}{e}.\psixj{j}{d}{k}{g}{a}{f}.\psixj{c}{d}{h}{k}{b}{j}.\end{multline}
\end{itemize}
\end{theorem}
Note that the $\e$-$6j$-symbols \eqref{6jsymb0} are only partially symmetric. The commutativity of the Pentagonal Diagram \eqref{pentdiag} for color modules shows up in the Elliot-Biedenharn identity. 
\begin{xca}\label{xcaUbar} (a) Why is the algebra $\bar{U}_\e$ not semisimple ? (Hint: otherwise every $\bar{U}_\e$-module would be semisimple, that is, a sum of simple modules.) 

(b) We have claimed the self duality of the modules $V_r$: determine explicitly an isomorphism $V_r^*\rightarrow V_r$ (use \eqref{dualmod} and the formulas in Section \ref{simpledef} !). 

\noindent (c) By Theorem \ref{singCG}, for any maximal trace zero submodule $U$ of $V_a\otimes V_b\otimes V_c$ which is a summand, the complementary submodule $W$ such that $V_a\otimes V_b\otimes V_c = W \oplus U$ is completely reducible. Show that any simple submodule $V$ of $V_a\otimes V_b\otimes V_c$ such that $V\not\subset U$ is a color module. Deduce that given any two maximal trace zero summands $U_i$ ($i=1,2$) of $V_a\otimes V_b\otimes V_c$, any simple submodule of $U_1$ is a submodule of $U_2$. 
 \end{xca}
Hence, for every maximal trace zero submodule $U$ of $V_a\otimes V_b\otimes V_c$, the matrix \begin{equation}\label{changemat} \left(\left\lbrace \begin{array}{ccc} a & b & l \\ c & k & j\end{array}\right\rbrace_\e\right)_{l,j}\end{equation}
of $\e$-$6j$-symbols relates the two basis of invariant maps $V_k \rightarrow V_a\otimes V_b\otimes V_c/U$ given by $\{(Y^l_{a,b}\otimes {\rm id}_{V_c})Y^k_{l,c}\}_l$ and $\{({\rm id}_{V_a}\otimes Y^j_{b,c})Y^k_{a,j}\}_j$.

\subsection{Pure regular Clebsch-Gordan decomposition and $6j$-symbols} Similarly to Theorem \ref{singCG} and the change of basis matrix \eqref{changemat}, the tensor products of simple regular $\Ue$-modules can be split in different ways, related by morphisms that we are going to define. 
\smallskip

First we consider the duals of regular $\Ue$-modules. Recall the coordinates \eqref{coordinates} of the set ${\rm Spec}(Z_\e)$ of central characters of $\Ue$, and the degree $n$ covering map over the regular ones, $$\tau\colon {\rm Spec}(Z_\e)\setminus \Dd \ra {\rm Spec}(Z_0)\setminus \{\pm {\rm id}\}.$$ Let $\chi \in {\rm Spec}(Z_\e)\setminus \Dd$. Denote by $V_{\chi}$ the corresponding simple $\Ue$-module, and let $\chi^{-1}$ be given by $$\tau(\chi^{-1}) = \tau(\chi)^{-1} \quad \mbox{and}\quad c_{\chi^{-1}} = c_{\chi}$$ where $\tau(\chi)^{-1}$ is the inverse of $\tau(\chi)$ in the group ${\rm Spec}(Z_0)\cong H$. From \eqref{group0} and \eqref{dualmod} we get: 
\begin{lemma}\label{dualmap} The dual module $V_\chi^*$ coincides with $V_{\chi^{-1}}$. Hence the duality of $\Ue$-Mod induces an isomorphism of the bundle $\Xi$ that lifts the inversion map on ${\rm Spec}(Z_0)\setminus \{\pm {\rm id}\}$. 
\end{lemma}
It is immediate to check that the regular simple $\Ue$-modules $V(\lambda,a,b)$ have vanishing quantum dimension, and so are trace zero modules. Thus, we are in some sense in a situation opposite to that of Section \ref{singcat}, where we dealt with the color modules $V_r$. What makes the tensor products of the regular $V(\lambda,a,b)$s easy to handle are the following facts, that we have proved in Lemma \ref{HopfZ0} and Theorem \ref{moduleteo1} (i):
\begin{itemize}
 \item (a) $Z_0$ is a Hopf subalgebra of $\Ue$, and in particular $\Delta(Z_0) \subset Z_0\otimes Z_0$;
\item (b) for all $\chi\in {\rm Spec}(Z_\e)\setminus \Dd$, the algebra $U_\e^{\chi}$ is simple, and isomorphic to $M_n(\mc)$.
\end{itemize}
Let $h \in {\rm Spec}(Z_0)\setminus \{\pm {\rm Id}\}$ be a regular $Z_0$-character, and $\Ii^h\subset \Ue$ the ideal generated by ${\rm Ker}(h)$. For any $\chi\in\tau^{-1}(h)$, the algebra $U_\e^{h} = \Ue/\Ii^{h}$ is isomorphic to $\Ue^{\chi}\otimes_h Z_\e$. Then, in virtue of (b) above, it is semisimple, with a direct product decomposition into complementary ideals
\begin{equation}\label{semisimple} U_\e^{h} = \prod_{\chi\in \tau^{-1}(h)} U_\e^{\chi}.
\end{equation}
Correspondingly, the unit $1\in U_\e^{h}$ can be written as
\begin{equation}\label{unitdec1}
 1 = \sum_{\chi\in \tau^{-1}(h)} e_{\chi},
\end{equation}
where the $e_{\chi}$s are the units of the subalgebras $U_\e^{\chi}\subset U_\e^{h}$, and satisfy $U_\e^{\chi} = U_\e^{h}e_{\chi}$, and $e_{\chi}e_{\chi'}=0$ for $\chi\ne \chi'$ \cite[Prop. XVII.4.3]{L}. 

Since $U_\e^{h}$ is semisimple, every $U_\e^{h}$-module is semisimple, that is, a sum of simple submodules. On another hand, because of (a) above, $Z_0$ acts by scalars on any tensor product $V$ of simple $\Ue$-modules, and so $V$ is naturally a $U_\e^{h}$-module for some $h\in {\rm Spec}(Z_0)$. In fact, \eqref{group0} shows that $h$ is equal to the product of the $Z_0$-characters of the tensorands. 

This applies in particular to the $\Ue$-modules $V_{\rho}\otimes V_{\mu}$ for all $\rho$, $\mu \in {\rm Spec}(Z_\e)\setminus \Dd$. They are $U_\e^{h}$-modules, where $h=\tau(\rho)\tau(\mu)$. If $h\ne \pm {\rm id}$, $U_\e^{h}$ is semisimple, and hence
\begin{equation}\label{splitmodule1} V_{\rho}\otimes V_{\mu} =  1. (V_{\rho}\otimes V_{\mu}) = \oplus_{\chi\in \tau^{-1}(h)}\  e_{\chi} . (V_{\rho}\otimes V_{\mu}),\end{equation} 
where the projectors $e_{\chi}$ map $V_{\rho}\otimes V_{\mu}$ onto a submodule isomorphic to $V_\chi$. We deduce the following analog of Theorem \ref{singCG} for regular $\Ue$-modules:
\begin{proposition} \label{CG} \textit{(i) Any tensor product of simple $\Ue$-modules which has a regular $Z_0$-character is a semisimple $\Ue$-module.} 

\textit{(ii) If $\rho$, $\mu \in {\rm Spec}(Z_\e)\setminus \Dd$ are such that $h=\tau(\rho)\tau(\mu) \ne \pm {\rm id}$, then} \begin{equation}\label{splitmodule2} V_\rho \otimes V_\mu \cong \oplus_{\chi\in \tau^{-1}(h)} V_{\chi}.\end{equation}
\end{proposition}
\begin{remark}\label{idempreg} The idempotents $e_{\chi}$ of $\Ue^h$ play in Proposition \ref{CG} the same role as the Jones-Wentzl idempotents \eqref{JW} do in Theorem \ref{singCG}. In both cases, the simple summands are distinguished by the Casimir element. For color modules, its values are determined by the classical or quantum dimension. For regular modules, these are constantly equal to $n$ and $0$, respectively; the Casimir element selects a $n$th-root of $z$ (see Theorem \ref{bundle} (ii)).
\end{remark}
\begin{definition} Let $h\in {\rm Spec}(Z_0)\setminus \{\pm {\rm id}\}$ be a regular $Z_0$-character. The \textit{multiplicity module} $M(h)$ is the $n$ dimensional vector space spanned by the idempotents $e_{\chi}$ of $\Ue^h$, where $\chi\in \tau^{-1}(h)$.

A tuple $(\rho_1,\ldots,\rho_p)$ of regular central characters $\rho_i\in {\rm Spec}(Z_\e)\setminus \Dd$ is \textit{regular} if for all $1\leq k\leq k+l\leq p$, we have $$\tau(\rho_k)\tau(\rho_{k+1})\ldots\tau(\rho_{k+l})\in {\rm Spec}(Z_0)\setminus \{\pm {\rm id}\}.$$
\end{definition} 
By Theorem \ref{moduleteo1} (i) we have a corresponding notion of regular tuples of $\Ue$-modules.
\begin{remark} \label{fixed} For any two $h$, $h'\in {\rm Spec}(Z_0)\setminus \{\pm {\rm id}\}$, the decomposition \eqref{unitdec1} provides canonical isomorphisms $M(h) \cong M(h')$. Also, Proposition \ref{CG} (ii) implies that any two regular pairs $(\rho,\mu)$ and $(\nu,\kappa)$ satisfying $h=\tau(\rho)\tau(\mu) = \tau(\nu)\tau(\kappa)$ give isomorphic $\Ue$-modules $V_\rho \otimes V_\mu$ and $V_\nu \otimes V_\kappa$.
\end{remark}
We can reorganize the direct sum $\oplus_{\chi\in \tau^{-1}(h)} V_{\chi}$ into a tensor product as follows. Let $V$ be the vector space underlying the modules $V_\chi$. Define an action of $\Ue$ on $V\otimes M(h)$ by extending linearly the formula
\begin{equation}\label{standardcop} a.(v\otimes e_{\chi}) = (a_{\chi}.v)\otimes e_\chi
\end{equation}
for all $a\in \Ue$ and $v\in V$, where $a_\chi\in \Ue^\chi$ is the coset of $a$, with its canonical action on $V$. We have a canonical isomorphism $\oplus_{\chi\in \tau^{-1}(h)} V_{\chi}\cong V\otimes M(h)$ of $\Ue$-modules, mapping $v\in V_\chi$ to $v\otimes e_\chi$. Then, for all regular pairs $(\rho,\mu)$ with $h=\tau(\rho)\tau(\mu)$, any isomorphism of the form \eqref{splitmodule2} defines an isomorphism of $\Ue$-modules
\begin{equation}\label{Fdef}\fonc{F(\rho,\mu)}{V\otimes V}{V\otimes M(h)}{v_1\otimes v_2}{\sum_{\chi\in \tau^{-1}(h)}e_{\chi}(v_1\otimes v_2) \otimes e_{\chi}}
\end{equation}
by putting the $\Ue$-module structure of $V_\rho\otimes V_\mu$ on $V\otimes V$; when this structure is clear from the context, we write $F$ for $F(\rho,\mu)$, and similarly for 
the inverse \textit{evaluation map}
\begin{equation}\label{Kdef}K(\rho,\mu)\colon V\otimes M(h)\lra V\otimes V.\end{equation} 
Let $\Delta_{(\rho,\mu)}\colon \Ue \ra \Ue^\rho\otimes\Ue^\mu$, $a\mapsto \Delta(a)$ mod$(\Ii^\rho \otimes \Ii^\mu)$. The formula \eqref{standardcop}  shows that $K(\rho,\mu)$ and $\Delta_{(\rho,\mu)}$ are equivalent data, related by
\begin{equation}\label{twistcop}
\Delta_{(\rho,\mu)}(a) =  K(\rho,\mu)(a\otimes 1)K(\rho,\mu)^{-1}
\end{equation}
as operators in ${\rm Aut}(V_\rho\otimes V_\mu)$.
\begin{remark} Since the modules $V_\chi$ are simple, ${\rm End}_{\Ue}(V_\chi)\cong \mc$, and so $K(\rho,\mu)$ is uniquely determined up to a scalar factor only. We can reduce this to an ambiguity modulo \textit{$n$th roots of unity} by requiring that ${\rm det}(K(\rho,\mu)) = 1$. This produces a degree $n$ polynomial in the coordinates of $\rho$, $\mu\in {\rm Spec}(Z_\e)\setminus \Dd$.\end{remark}
We are now ready to define the operators and functional relations describing the isomorphisms 
$$(V_\rho\otimes V_\mu)\otimes V_\nu \cong V_\rho\otimes (V_\mu\otimes V_\nu) \cong 
\oplus_{\chi\in \tau^{-1}(h)} (\underbrace{V_{\chi}\oplus \ldots \oplus V_{\chi}}_{n\ \mathrm{times}})$$
for regular triples $(\rho,\mu,\nu)$, where $h=\tau(\rho)\tau(\mu)\tau(\nu)$.
\begin{definition}\label{6jdef} Let $(\rho,\mu,\nu)$ be a regular triple of $\Ue$-modules. Put $f=\tau(\rho)$, $g = \tau(\mu)$, $h=\tau(\nu)$. The {\it regular $6j$-symbol operator} of $(\rho,\mu,\nu)$ is the linear isomorphism 
\begin{equation}\label{Rdef} \Rr(\rho,\mu,\nu)\colon M(fgh) \otimes M(gh) \lra M(fg)\otimes M(fgh)
\end{equation}
that makes the following diagram commutative:
\begin{equation}\label{fusiondiag} \xymatrix{ V \otimes M(fgh) \otimes M(gh) \ar[rr]^{{\rm id} \otimes \Rr(\rho,\mu,\nu)} \ar[d]_{({\rm id}\otimes \Delta_0)(K)} & &  V \otimes M(fg) \otimes M(fgh) \ar[d]^{(\Delta_0\otimes {\rm id})(K)} \\ V_\rho \otimes V \otimes M(gh) \ar[d]_{{\rm id}\otimes K} & &  V\otimes M(fg) \otimes V_\nu \ar[d]^{K \otimes {\rm id}}\\ V_\rho\otimes (V_\mu\otimes V_\nu) & & (V_\rho\otimes V_\mu)\otimes V_\nu \ar[ll]^{a_{V_\rho,V\mu,V_\nu}}}\end{equation}
where $\Delta_0(x) = x\otimes 1$ (the \textit{standard} coproduct). In operator form, we have
\begin{equation}\label{modpent}
K_{12}(\rho,\mu)K_{13}(\chi_r,\nu)\Rr_{23}(\rho,\mu,\nu) = K_{23}(\mu,\nu)K_{12}(\rho,\chi_l)
\end{equation}
for all $\chi_r \in \tau^{-1}(fg)$, $\chi_l\in \tau^{-1}(gh)$.
\end{definition}
Thus, $\Rr$ determines via \eqref{modpent} the associativity constraint $a$ of $\Ue$-Mod over regular modules, like $K$ determines via \eqref{twistcop} the tensor product of $\Ue$-Mod over regular modules. The Pentagonal Diagram \eqref{pentdiag} translates as follows:
\begin{proposition}\label{3cocycloid} \textit{Let $(\kappa,\rho,\mu,\nu)$ be a regular $4$-tuple of $\Ue$-modules. Set $f=\tau(\kappa)$, $g = \tau(\rho)$, $h=\tau(\mu)$, $k=\tau(\nu)$. The following diagram is commutative:}
$$\xymatrix{M(fghk) \otimes M(gh) \otimes M(ghk) \ar[dd]^{(\Delta_0\otimes {\rm id})(\Rr)} & & M(fghk)\otimes M(ghk) \otimes M(hk)  \ar[ll]_{{\rm id} \otimes \Rr(\rho,\mu,\nu)} \ar[d]^{({\rm id}\otimes \Delta_0)(\Rr)} \\ & & M(fg)\otimes M(fghk) \otimes M(hk) \ar[d]^{(\tau\Delta_0\otimes {\rm id})(\Rr)}\\  M(fgh) \otimes M(gh) \otimes M(fghk) \ar[rr]^{\Rr(\kappa,\rho,\mu)\otimes {\rm id}} & & M(fg) \otimes M(fgh)\otimes M(fghk) }$$
\textit{where $\tau$ is the flip map. In operator form, we have the \emph{Pentagon Equation}:}
\begin{equation}\label{penteq} \Rr_{12}(\kappa,\rho,\mu)\Rr_{13}(\kappa,\chi_1,\nu)\Rr_{23}(\rho,\mu,\nu) = \Rr_{23}(\chi_2,\mu,\nu)\Rr_{12}(\kappa,\rho,\chi_3)
\end{equation}
for all $\chi_1 \in \tau^{-1}(gh)$, $\chi_2\in \tau^{-1}(fg)$, $\chi_3\in \tau^{-1}(hk)$.
\end{proposition}
\begin{proof} Commutativity of the diagram is equivalent to \eqref{penteq}. The latter is proved by a straightforward computation using \eqref{modpent}; details are left as an exercise (or, compare with the proof of Proposition \ref{propcan} iii)).
\end{proof}
\begin{remark}\label{borelcase} \textit{(The Borel case)} The notion of regular module makes sense as well for any Borel subalgebra $U_\e b$ of $\Ue$, say the positive one, generated by $K^{\pm 1}$ and $E$. Since $U_\e b$ has no Casimir element, the isomorphism classes of simple regular $U_\e b$-modules correspond under the map $\varphi$ of \eqref{chi} to triangular matrices up to sign in $PSL_2\mc \setminus \{\pm {\rm Id}\}$. Then, for regular pairs $(\rho,\mu)$ of simple $U_\e b$-modules, \eqref{splitmodule2} simplifies to a splitting into $n$ isomorphic copies of a single regular $U_\e b$-module $V_{\rho\mu}$. The multiplicity module $M(h)$ becomes the space of equivariant projections ${\rm Hom}_{U_\e\! b}(V_\rho\otimes V_\mu,V_{\rho\mu})$, and \eqref{Kdef} is the map $$\fonc{K}{V_{\rho\mu}\otimes {\rm Hom}_{U_\e\! b}(V_{\rho\mu},V_\rho\otimes V_\mu )} {V_\rho\otimes V_\mu }{v\otimes i}{i(v).}$$ 
\end{remark}
\begin{xca} Let $V$ be a finite dimensional vector space and $f$, $g\in {\rm End}(V)$ such that $f^2=f$, $g^2=g$ and $fg = gf$. Show that $R:=f\otimes g$ satisfies \begin{equation}\label{constpent} R_{12}R_{13}R_{23} = R_{23}R_{12}\end{equation}
in ${\rm End}(V\otimes V\otimes V)$. Show that if $R:=f\otimes {\rm id}$ or $R:={\rm id} \otimes f$ is a solution of \eqref{constpent}, then $f^2=f$.   
\end{xca}
\section{The regular $6j$-symbols as bundle morphisms}\label{6JMORPH}
The FRT method associates a cobraided Hopf algebra to any finite dimensional invertible solution of the Quantum Yang-Baxter Equation (see eg. \cite[Ch. VIII]{K}). Similarly, any finite dimensional invertible solution of the \textit{constant} Pentagon Equation \eqref{constpent} is the canonical element of the  \textit{Heisenberg double} of some Hopf algebra \cite{BS,Mi,D}. 

We are going to see how this result can be adapted to equation \eqref{penteq}, which has the form of a (non-Abelian) $3$-cocycle identity over the group $PSL_2\mc$ via the map $\sigma\tau\colon {\rm Spec}(Z_\e) \ra PSL_2\mc$. We proceed in several steps to identify both the evaluation map \eqref{Kdef} and the regular $6j$-symbol operator \eqref{Rdef} as instances of a same bundle morphism $\Rr \colon \Xi^{(2)} \lra \Xi^{(2)}$.

\subsection{The QUE algebra $U_h$} The \textit{quantum universal envelopping (QUE) algebra} $U_h = U_h(sl_2)$ is the Hopf algebra over $\mc[[h]]$ topologically generated by three variables $X$, $Y$ and $H$ with relations 
$$[H,X]=2X\ ,\ [H,Y]=-2Y\ ,\ [X,Y] = \frac{\sinh(hH/2)}{\sinh(h/2)} = \frac{e^{hH/2}-e^{-hH/2}}{e^{h/2}-e^{-h/2}}.$$
(See eg. \cite[Ch. 6--8]{CP} and \cite[Ch. XVII]{K}.) The comultiplication and counit are determined by
$$ \Delta(H) = H\otimes 1+1\otimes H,$$ $$\Delta(X) = \ X\otimes e^{hH/4} + e^{-hH/4}\otimes X\ ,\ \Delta(Y) = Y\otimes e^{hH/4} + e^{-hH/4}\otimes Y$$
and 
$$\eta(H) = \eta(X) = \eta(Y) = 0.$$
The antipode is $$S(H) = -H\ ,\ S(X) = -e^{hH/2}X\ ,\ S(Y) = -e^{-hH/2}Y.$$
The QUE algebra $U_h$ is a \textit{topological deformation} of the universal enveloping algebra $U sl_2$, in the sense that $U_h\cong Usl_2[[h]]$ as $\mc[[h]]$-modules, and $U_h/hU_h \cong Usl_2$ as Hopf algebras, where $U sl_2$ has the Hopf algebra structure determined by \begin{equation}\label{UEAdef} \Delta(x) = x\otimes 1+1\otimes x\ ,\ \eta(x)=0\ ,\ S(x)=-x\end{equation} for all $x\in sl_2$. In particular, $U_h$ is equipped with the $h$-adic topology, and is a \textit{topologically} free $\mc[[h]]$-module; the algebraic tensor product $U_h \otimes U_h$ is thus equally completed in $h$-adic topology. Then, the coproduct $\Delta$ is well-defined, and all the above structure maps are continuous. 

Let us identify $\Uq$ with the algebra $U_q'$ of Exercise \ref{UEA}. The latter is explicitly defined for all values of $q$ and inherits from $U_q$ a structure of Hopf algebra. There is an injective morphism of Hopf algebras $i\colon \Uq' \ra U_h$, given by $i(q)=e^{h/2}$ and 
\begin{equation}\label{embedalg} \begin{array}{lll}i(E) = Xe^{hH/4} & ,& i(F) =e^{-hH/4}Y\ ,\\ i(K^{\pm 1}) = e^{\pm hH/2}& ,& i(L) = XY-e^{-hH/4}YXe^{hH/4}.\end{array}
\end{equation}
Hence, for what regards properties independent of the specific evaluation $q=\e$, we will consider $\Uq$ as a subHopf algebra of $U_h$.
\smallskip

\subsubsection{The QUE dual $U_h^\circ$} Duality is a delicate question for infinite dimensional algebras. We say that two Hopf algebras $A$ and $A'$ over a ground ring $k$ are {\it dual} if there exists a bilinear pairing $\langle \ , \ \rangle\colon A\otimes A' \ra k$ which is non degenerate in the sense that for all $f\in A'$ (resp. $u\in A$), $\langle u,f\rangle = 0$ for all $u\in A$ (resp. $f\in A'$) implies $f=0$ (resp. $u=0$), and
\begin{equation}\label{propertiesdual} \begin{array}{c}\langle u, fg\rangle  = \langle \Delta(u), f\otimes g\rangle\ ,\ \langle u\otimes v,\Delta(f) \rangle = \langle uv, f\rangle\\ \langle u, S(f)\rangle = \langle S(u), f\rangle\ ,\ \eta(f) = \langle 1 , f\rangle\ ,\ \eta(u) = \langle u , 1\rangle\end{array} \end{equation}
for all $u$, $v\in A$, $f$, $g\in A'$, where we denote by the same letters the structure maps of $A$ and $A'$. When suitably interpreted, these formulas give indeed the linear dual $U_h^*={\rm Hom}_{\mc[[h]]}(U_h,\mc[[h]])$ a structure of topological Hopf algebra dual to $U_h$ for the natural pairing $\langle \ , \ \rangle\colon U_h\otimes U_h^* \ra \mc[[h]]$ 
\cite[Def. 6.3.3]{CP}. However $U_h^*$ is not a QUE algebra since there does not exist any Lie algebra $\mathfrak{g}$ such that $U_h^*/hU_h^* \cong U\mathfrak{g}$ as Hopf algebras (otherwise $U_h^*$ should be cocommutative up to first order by \eqref{UEAdef}). 

A way to repair this inconvenience is to consider the space
$$U_h^\circ = \sum_{l\geq 0} h^{-l} I^l,$$
where $\textstyle I = \sum_{i,j,k} (X^iY^jH^k)^* U_h^*$ is the maximal ideal of $U_h^*$, the elements $(X^iY^jH^k)^*$ being dual to PBW basis vectors of $U sl_2$ and the sums completed in $h$-adic topology. Recall from Corollary \ref{PoissonH} and Remark \ref{dualLPG} that the Poisson-Lie structure of $H\cong {\rm Spec}(Z_0)$ given by the bracket $\psi_*\{\ ,\ \}$ is dual to the so called standard one on $PSL_2\mc$. We have:
\begin{proposition}\label{QUEdual1} \cite[Prop. 10.3]{ES} \emph{$U_h^\circ$ is a Hopf algebra dual to $U_h$ under the completion of the natural pairing $\langle \ , \ \rangle\colon U_h\otimes U_h^* \ra \mc[[h]]$. It is isomorphic to the QUE algebra $U_h(\mathfrak{h})$, where $\mathfrak{h}$ is the Lie algebra of $H$ with bialgebra structure tangent to $\psi_*\{\ ,\ \}$}. 
\end{proposition}
We call $U_h^\circ = U_h \mathfrak{h}$ the \textit{QUE dual} of $U_h$. 
\smallskip

\subsubsection{The finite dual} There is another notion of dual Hopf algebra that we will meet in the sequel. Its definition makes sense for arbitrary Hopf algebras (see \cite[$\S$ 9.1]{Mo} or \cite[Ch. 3, Prop. 1.1.3]{KS}): 
\begin{definition}\label{defdualHopf} The \emph{finite dual} of a $k$-algebra $A$ is the subspace $A^\bullet$ of the linear dual $A^*$ defined by
$$A^\bullet = \{ f\in A^*\ |\ f(I) = 0 \mathrm{\ for \ some\ ideal\ } I \mathrm{\ such \ that\ } {\rm dim}_\mc (A /I) <\infty \}.$$ 
 \end{definition}
We have:
\begin{proposition}\label{dualisHopf} \textit{If $A$ is a Hopf algebra, the finite dual $A^\bullet$ is a Hopf algebra, with structure maps defined by \eqref{propertiesdual} under the natural pairing $\langle \ , \ \rangle\colon A\otimes A^* \ra k$. Moreover, $A^\bullet$ is the largest subspace $V$ of $A^*$ with coproduct in $V\otimes V$.}
\end{proposition}
Note that the coproduct of $A$ always gives $A^*$ an algebra structure; only the product may cause some trouble, as its adjoint for $\langle\ ,\ \rangle$ may not map into the subspace $A^*\otimes A^* \subset (A\otimes A)^*$. 

The finite dual $A^\bullet$ can be equivalently defined as the subalgebra of $A^*$ generated by the \textit{matrix elements} of all finite dimensional $A$-modules $V$, that is, by the linear functionals
\begin{equation}\label{me} \fonc{c^V_{l,v}}{A}{\mc}{a}{l(a.v)}\end{equation}
where $v\in V$, $l\in V^*$. In these terms, Proposition \ref{dualisHopf} follows from the fact that the category of finite-dimensional $A$-modules is closed under taking duals, direct sums and tensor products. 

A topological interpretation of $A^\bullet$ is as follows. Let
$$\Jj = \{ {\rm Ker}(\rho)\ |\ \rho\colon A\ra {\rm End}(V) \ \mbox{is a finite dimensional representation} \}.$$
Define the \textit{$\Jj$-adic topology} on $A$ by taking as a base of neighborhoods of $a\in A$ the sets $\{a+J\ |\ J\in \Jj\}$. Similarly, define a topology on $A\otimes A$ by taking as a base of neighborhoods of $a\otimes b\in A\otimes A$ the sets $\{a\otimes b + L\ |\ L = A\otimes J+K\otimes A \ \mbox{and}\  J, K\in \Jj\}$. 
Then $A$ is a topological algebra (all the structure maps are continuous), and if the ground ring $k$ of $A$ is given the discrete topology, $A^\bullet$ is the set of continous $k$-linear maps $A\ra k$. For $A=U_h$, by taking representations $\rho$ on free $\mc[[h]]$-modules of finite rank, the finite dual $U_h^\bullet$ is the set of $\mc[[h]]$-linear maps $U_h \ra \mc[[h]]$ which are continuous for the $\Jj$-adic topology. It is usually called the \textit{quantized function ring}, and denoted by $SL_h(2)$ \cite[Th. 7.1.4]{CP}. 
\smallskip

The finite dual $U_q^\bullet$ is defined also for all values of $q\in \mc$, $q\ne -1$, $0$, $1$. At $q=\e$ it is very big, as it contains the representative functions on the bundle $\Xi_A$ of Theorem \ref{bundle2} (compare eg. with Theorem \ref{classdual} below). So $U_q^\bullet$ is usually defined in the litterature by restricting the matrix elements \eqref{me} to the $U_q$-modules $V_{r,q}^\pm$; then it coincides with the rational form $SL_q(2)$ of $SL_h(2)$ (see \cite[Ch. 13]{CP}, \cite[Ch. VII]{K}, \cite[Ch. 3]{KS}).
\subsection{The Heisenberg double $\Hh_h = \Hh(U_h)$} The Heisenberg double of $U_h$ is a topological $\mc[[h]]$-module isomorphic to $U_h \otimes U_h^\circ$, with an algebra structure given by a {\it smash product}. Let us recall this notion (see \cite{Mo}).
\begin{definition}\label{smashdef} Let $A$ be a Hopf $k$-algebra and $V$ a left $A$-module which is simultaneously an algebra. We say that $V$ is an {\it $A$-module algebra} if both its product $V\otimes V\ra V$ and unit $k\ra V$ are morphisms of $A$-modules. The \emph{smash product} $V \ \sharp \ A$ of $A$ and an $A$-algebra $V$ is the algebra isomorphic to $V\otimes A$ as a vector space, with product given by
$$(u \ \sharp \  a)(v \ \sharp \  b) = u(a_{(1)}.v)\ \sharp \  a_{(2)}b$$
for all $u$, $v\in V$ and $a$, $b\in A$, where we put $\Delta(a) = a_{(1)} \otimes a_{(2)}$ (Sweedler's sigma notation).
\end{definition}
The smash product $V \ \sharp \ A$ encodes the commutation relations between the linear operators induced by the left action of $V$ on itself, and the linear operators induced by the action of $A$ on $V$. Indeed, one checks without difficulty that:
\begin{lemma}\label{Heisrep} The map
\begin{equation}\label{Heisrepform}
 \fonc{\lambda}{V \ \sharp \ A}{{\rm End}(V)}{v\ \sharp \ a}{\left(u\mapsto v (a.u)\right)}
\end{equation}
defines a representation of $V \ \sharp \ A$ (the \emph{Heisenberg representation}).
\end{lemma}
\begin{definition} \label{Heisdef} Let $\langle \ , \ \rangle\colon A\otimes A' \ra k$ be a duality of Hopf $k$-algebras $A$ and $A'$. The \textit{Heisenberg double} $\Hh (A)$ is the smash product $A\ \sharp \ A'$, where $A$ is made into an $A'$-module via the \emph{left regular action} 
\begin{equation}\label{leftregact}\fonc{\rightharpoonup }{A' \otimes A}{A}{x\otimes a}{x\rightharpoonup a = a_{(1)}\langle a_{(2)},x\rangle.}
\end{equation}
\end{definition}
For Heisenberg doubles, the fact that \eqref{Heisrepform} defines a representation follows from the commutation relation
\begin{equation}\label{commuterel}(1\ \sharp \ x)(a\ \sharp \ 1) = (a_{(1)}\ \sharp\ 1)(1\ \sharp \ \langle a_{(2)},x_{(1)}\rangle x_{(2)})
\end{equation}
for all $a\in A$, $x\in A'$. In particular, the Heisenberg double $\Hh_h = \Hh (U_h)$ is the topological algebra isomorphic to $U_h \otimes U_h^\circ$ as a $\mc[[h]]$-module, with product given by
$$(u \ \sharp \  x)(v \ \sharp \  y) = u(x_{(1)}\rightharpoonup v)\ \sharp \  x_{(2)}y = \langle v_{(2)},x_{(1)}\rangle u v_{(1)} \ \sharp \  x_{(2)} y$$
for all $u$, $v\in U_h$, $x$, $y\in U_h^\circ$. 
\smallskip

\subsubsection{A classical example: the cotangent bundle $T^*G$.} The Heisenberg double $\Hh_h$ may be understood as a deformation of the following classical geometric situation. Let $G$ be an affine algebraic group over $\mc$. There are three ``classical'' Hopf algebras associated to $G$: 
\begin{itemize}
 \item The group algebra $\mc G$, where $\Delta(g) = g\otimes g$, $\eta(g) = 1$, and $S(g) = g^{-1}$ for all $g\in G$.
\item The coordinate ring $\mc[G]\subset \mc G^\bullet$ (as the subset of algebraic maps).
\item The universal envelopping algebra $U\mathfrak{g}$, with structure maps \eqref{UEAdef}.
\end{itemize}
Note that the adjoint action of $G$ on $\mathfrak{g}$ extends linearly to give an action of $\mc G$ on $U\mathfrak{g}$ by Hopf algebra automorphisms. Denote by $e$ the identity element of $G$, and identify $U\mathfrak{g}$ with the space of all left invariant differential operators on $G$. For simplicity, assume that $G$ is connected, simply connected, and semisimple. We have:
\begin{theorem}\label{classdual} \emph{(See \cite[Th. 4.3.13]{A})} The pairing \begin{equation}\label{classpair} \fonc{\langle \ ,\ \rangle}{U\mathfrak{g} \times \mc[G]}{\mc}{X \otimes f}{(X.f)(e)}\end{equation} induces an injective morphism $U\mathfrak{g} \ra \mc[G]^\bullet$ and an isomorphism $\mc[G]\ra U\mathfrak{g}^\bullet$. The Hopf algebra $\mc[G]^\bullet$ is generated by $U\mathfrak{g}$ and the evaluation maps $f\mapsto f(g)$, $g\in G$. More precisely, $\mc[G]^\bullet \cong U\mathfrak{g} \otimes \mc G$ as a coalgebra, and its product and antipode are given by
\begin{equation}\label{smash1} (X \otimes  f)(Y \otimes  g) = X(f_{(1)}.Y)\ \otimes \  f_{(2)}g,\end{equation} $$S(X\otimes f) = S(f_{(1)}).S(X)\otimes S(f_{(2)}).$$
\end{theorem}
In the situation of Theorem \ref{classdual}, the action of $U\mathfrak{g}$ on $\mc[G]\cong U\mathfrak{g}^\bullet$ by left invariant derivations takes a form dual to \eqref{leftregact}, since
\begin{equation}\label{actclas1}(X.f)(a) = \frac{{\rm d}}{{\rm dt}}\left( f(ae^{tX})\right)_{|t=0} = f_{(1)}(a)\langle X,f_{(2)}\rangle\end{equation}
for all $X\in \mathfrak{g}$, $f\in \mc[G]$ and $a\in G$, where $\langle\ ,\ \rangle$ is as in \eqref{classpair}. Hence, by dualizing Definition \ref{Heisdef} and considering $\mc[G]$ as an $U\mathfrak{g}$-module via \eqref{actclas1}, we find that $\Hh(\mc[G])=\mc[G] \ \sharp\ U\mathfrak{g}$ coincides with the algebra of all differential operators on $G$, with its usual action \eqref{Heisrepform} on $\mc[G]$. On another hand, the symmetrization map $$\fleche{S\mathfrak{g}}{U\mathfrak{g}}{x_1\ldots x_r}{\sum_{\sigma\in S_r}x_{\sigma(1)}\ldots x_{\sigma(r)}}$$ yields vector space isomorphisms $\mc[\mathfrak{g}^*] \cong S\mathfrak{g}\cong U\mathfrak{g}$ that allow one to identify $\Hh(\mc[G])$ with the space of functions on the cotangent bundle $T^*G \cong G\ltimes \mathfrak{g}^*$. The algebra structure is determined by the pairing \eqref{classpair}, and hence by the canonical symplectic structure on $T^*G$ via the map assigning to each function its Hamiltonian vector field.  
\smallskip

\subsubsection{$\Hh_h$ as a deformation of $\mc[T^*G]$.} The cotangent bundle $T^*G$ is an example of Drinfeld double Lie group. It is associated to the \textit{trivial} Poisson-Lie structure of $G$. By starting with the \textit{standard} Poisson-Lie structure of $PSL_2\mc$ the general theory provides us with a Drinfeld double $\Dd(PSL_2\mc)$ diffeomorphic to $PSL_2\mc \times H$ in a neighborhood of the identity element $e$. It carries a Poisson bracket $\{\ ,\ \}_s$ which is non degenerate at $e$, and given by (see \cite[Prop. 6.1.16]{KS})
\begin{equation} \label{formsymp} \{f_1,f_2\}_s = \sum_i (\partial_i f_1\partial^i f_2-\partial^i f_1\partial_i f_2) + \sum_i (\partial_i' f_1(\partial^i)' f_2-(\partial^i)' f_1\partial_i'f_2).
\end{equation}
Here, $\partial_i$ and $\partial^i$ (resp. $\partial_i'$ and $(\partial^i)'$) are dual basis of right (resp. left) invariant vector fields on $PSL_2\mc$ and $H$, respectively, considered as vector fields on $\Dd(PSL_2\mc)$. The function space $\mc[\Dd(PSL_2\mc)]$ inherits from $\{\ ,\ \}_s$ a structure of Heisenberg double algebra $\Hh(\mc[PSL_2\mc])$ that may be deformed to $\Hh(U_h^\bullet)$ (\cite[Prop. 3.3]{STS}, \cite[Th. 3.10]{Lu2}). More precisely, by working dually and letting $\mathfrak{d}$ denote the Lie algebra of $\Dd(PSL_2\mc)$, one obtains:
\begin{proposition}\label{STSteo} \textit{The Heisenberg double $\Hh_h =\Hh (U_h)$ is a topological algebra deformation of $U\mathfrak{d}$ over $\mc[[h]]$, and is a quantization of $(\Dd(PSL_2\mc),\{\ ,\ \}_s)$.}
\end{proposition}
The last claim means that the Poisson bracket $\{\ ,\ \}_s$ determines a bivector $\pi_s\in \mathfrak{d}\otimes \mathfrak{d}$ such that $r= 1+h \pi_s$ mod$(h^2)$, where \begin{equation}\label{canpre} r= \lambda^{-1}({\rm id}_{U_h})                                                                                                                                                                                                                                                                                                           \end{equation}
and $\lambda$ is the canonical map $\lambda\colon U_h \otimes U_h^\circ \lra {\rm Hom}_{\mc[[h]]}(U_h)$. 
\begin{remark} As suggested by $T^*G$ and $\Dd(G)$ above, the quantum Heisenberg doubles (resp. \eqref{canpre}) are closely related to Fourier duality and quantum Drinfeld doubles (resp. universal $R$-matrices). We refer to \cite{STS} and \cite{Lu,CR} for results in this direction.\end{remark}
\subsection{The canonical morphism $\Rr\colon \Xi^{(2)} \ra \Xi^{(2)}$} Recall the bundle $\Xi_M$ of Theorem \ref{bundle2}. Let ${\rm Spec}(Z_\e)^2_{reg}$ denote the set of regular pairs of central characters. Define a new bundle
\begin{equation}\label{XI2} \Xi^{(2)}\colon M_\e^{ (2)} \ra {\rm Spec}(Z_\e)^2_{reg}
\end{equation}
by restricting the base space of the product bundle $\Xi_M\times \Xi_M$ and regarding each fiber as a tensor product of $\Ue$-modules; hence, for any $(\rho,\mu)\in {\rm Spec}(Z_\e)^2_{reg}$, the fiber over $(\rho,\mu)$ has the form $V_\rho \otimes V_\mu$. By Remark \ref{fixed}, the fibers over any two pairs $(\rho,\mu),(\nu,\kappa)\in {\rm Spec}(Z_\e)^2_{reg}$ such that $\tau(\rho)\tau(\mu) = \tau(\nu)\tau(\kappa)$ are isomorphic $\Ue$-modules.
\begin{proposition}\label{actext2} The coproduct induces an action $\Delta_\Gg$ of $\Gg$ on $\Xi^{(2)}$ by bundle morphisms. 
\end{proposition}
\begin{proof} As in Theorem \ref{bundle2} it is enough to show that the coproduct induces a homomorphism $\Gg \ra  {\rm Aut}(\hat{U}_\e \otimes \hat{U}_\e)$ whose image preserves $Z_\e\otimes Z_\e$. 
For all $x\in \Ue$ one computes in $U_q$ that 
\begin{align*} \Delta (\left[\frac{E^n}{[n]!},u\right]) & =   \left[\frac{E^n}{[n]!}\otimes 1+ K^n\otimes \frac{E^n}{[n]!}+\ldots ,\Delta(u)\right]\\
& = \left[\frac{E^n}{[n]!},u_{(1)}\right]\otimes u_{(2)} + \left[\frac{K^n}{[n]!}, u_{(1)}\right] \otimes E^n u_{(2)} + \\
& \hspace{5cm} + u_{(1)}K^n\otimes \left[\frac{E^n}{[n]!},u_{(2)}\right]+\ldots\end{align*}
where we put $\Delta(u) = u_{(1)} \otimes u_{(2)}$, and the dots refer to elements which are vanishing at $q=\e$. By specializing at $q=\e$ and using \eqref{der0} we get
\begin{equation}\label{coactGXi}\Delta(D_e(u)) = (D_e \otimes 1 + z \otimes D_e - D_z \otimes z x)\Delta(u).\end{equation}
The right hand side is a derivation of $\hat{U}_\e \otimes \hat{U}_\e$ preserving $Z_\e\otimes Z_\e$. As in the proof of Proposition \ref{qca}, it can be integrated to a $1$-parameter group of automorphisms $\Delta (\exp(tD_e))$. Similar facts hold true for $\Delta(D_f(u))$ as well. We let $\Delta_\Gg$ be generated by the actions of $\Delta (\exp(tD_e))$  and $\Delta (\exp(tD_f))$, $t\in \mc$. 
\end{proof} 
We wish to let the element \eqref{canpre} act on the left on $\Xi^{(2)}$, or equivalently, by conjugation on the bundle of algebras obtained from $\Xi^{(2)}$ by replacing $\Xi_M$ with $\Xi_A$. In order for this to make sense, we realize it as an element of $\Hh_h\otimes \Hh_h$:  
\begin{definition} The \emph{canonical element of $\Hh_h$} is $R = \tau \circ (i \otimes i') (r)_{21} \in \Hh_h \otimes \Hh_h$, where the subscript ``$_{21}$'' means that the tensorands are permutated, and $i\colon U_h \ra \Hh_h$ and $i'\colon U_h^\circ \ra \Hh_h$ are the natural linear embeddings in $\Hh_h\cong U_h \otimes U_h^\circ$.\end{definition}
By writing $r = \sum_i e_i\otimes e^i$ in dual topological basis $\{ e_i\}$ and $\{e^i\}$ of $U_h$ and $U_h^\circ$, we have
\begin{equation}\label{express} R = \sum_i (1\ \sharp \ e^i) \otimes (e_i \ \sharp\ 1).\end{equation}
The next results sums up the fundamental properties of $R$ (compare with \cite[Th. IX.4.4]{K} and \cite[Prop. 1.2]{BS}). To simplify notations, we will often identify $u\in U_h$ with $i(u) = u\ \sharp \ 1$ and $x\in U_h^\circ$ with $i'(x) = 1\ \sharp \ x$.
\begin{proposition}\label{propcan}\textit{ i) We have} $R^{-1} = (S\otimes {\rm id})(R)$ and $({\rm id}\otimes \Delta)(R) = R_{12}R_{13}$.

\noindent \textit{ii) The following identities hold true:}
\begin{equation}\label{twistcop2bis}
 R(u \otimes 1) = \Delta(u)R\ ,\ (1\otimes x)R = R \Delta(x)
\end{equation}
\textit{for all $u\in U_h$ and $x\in U_h^\circ$, and} 
\begin{equation}\label{constpent2} R_{12}R_{13}R_{23} = R_{23}R_{12}.\end{equation}
\noindent \textit{iii) Denote by $\mu$ the product of $U_h$. The image $\Rr = (\lambda \otimes  \lambda)(R) \in {\rm Aut}_{\mc[[h]]}(U_h \otimes U_h)$ of the canonical element under the Heisenberg representation \eqref{Heisrepform} is given by} \begin{equation}\label{autprop} \Rr = ({\rm id}\otimes \mu)(\Delta \otimes {\rm id}).\end{equation}
\end{proposition}
\begin{proof} i) We have $R (S\otimes {\rm id})(R) = \sum_{i,j}(1\ \sharp\ e^i(e^j\circ S)) \otimes (e_ie_j\ \sharp\ 1)$. Then, by using \eqref{propertiesdual} we see that for all $x$, $y\in U_h^\circ$ and $u$, $v\in U_h$ we have
\begin{align}
\langle R (S\otimes {\rm id})(R), x\otimes u\otimes y\otimes v\rangle & = x(1)\eta(v)\sum_{i,j,(u)} y(e_ie_je^i(u_{(1)})e^j(S(u_{(2)}))) \notag \\ & = x(1)\eta(v) y(u_{(1)}S(u_{(2)}))\notag \\ & = x(1)\eta(v) y(1) \eta(u) \notag\\ & = \langle 1\otimes 1\otimes 1\otimes 1,x\otimes u\otimes y\otimes v\rangle.\notag
\end{align}
Here, $\eta$ is the counit of $U_h$, and $\langle\ ,\ \rangle$ the product of the natural duality pairings. It follows that $R^{-1} = (S\otimes {\rm id})(R)$. As for the second identity, note that we have $(\langle u,\ \rangle \otimes {\rm id})(R) = u$, $({\rm id}\otimes \langle\ ,x\rangle)(R) = x$, and similarly for $y$. Hence
$$\langle u\otimes x \otimes y,({\rm id} \otimes \Delta)(R)\rangle = \langle x\otimes y, \Delta(u)\rangle.$$
On another hand $\langle u\otimes x \otimes y,R_{12}R_{13}\rangle = \langle ({\rm id}\otimes \langle\ ,x\rangle)(R) ({\rm id}\otimes \langle\ ,y\rangle)(R),u\rangle = \langle xy,u\rangle$. The result follows again from \eqref{propertiesdual}.

ii) For all $u\in U_h$ and $x\in U_h^\circ$, we have $({\rm id}\otimes \langle \ ,x \rangle)(R(u \otimes 1)) = xu$, and
\begin{align}({\rm id}\otimes \langle \ ,x \rangle)(\Delta(u)R) & = \sum_i u_{(1)}e^i\langle u_{(2)}e_i,x\rangle \notag \\ & = \sum_i u_{(1)}e^i\langle u_{(2)},x_{(1)} \rangle\langle e_i,x_{(2)}\rangle \notag\\ & = u_{(1)}\langle u_{(2)},x_{(1)} \rangle x_{(2)} \notag .\end{align}
Together with \eqref{commuterel} this proves the first identity in \eqref{twistcop2bis}. The second one is similar. Finally, $R_{23} R_{12} = ({\rm id}\otimes \Delta )(R)R_{23}= R_{12} R_{13} R_{23}$ by \eqref{twistcop2bis} and i). 

iii) One checks \eqref{autprop} by a straighforward computation, writing $R$ as in \eqref{express}. Details are left as an exercise. 
\end{proof}
Like \eqref{embedalg} we have an embedding of the rational form $U_q\mathfrak{h}$ of $U_h^\circ = U_h\mathfrak{h}$ into $U_h^\circ$, and hence of the smash product $\Hh_q = U_q \ \sharp\  U_q\mathfrak{h}$ into $\Hh_h$. One can check that the conjugation action of $R$ on $\Hh_h\otimes \Hh_h$ induces an (outer) automorphism $\Rr$ of the subalgebra $\Hh_q\otimes \Hh_q$. Moreover, in complete analogy with the semi-classical situation considered in \cite{WX}, $\Hh_q$ has a double structure of quantum groupoid through which $\Rr$ factors to define an automorphism of $\Uq \otimes \Uq$. Then, by specializing at $q=\e$ and considering the resulting left action on tensor products of simple $\Ue$-modules, using \eqref{twistcop2bis}-\eqref{constpent2}, \eqref{twistcop} and \eqref{modpent} we get: 
\begin{theorem} \label{Rmorphism} \textit{The canonical element $R$ of $\Hh_h$ induces a bundle morphism $\Rr \colon \Xi^{(2)} \lra \Xi^{(2)}$ such that}
\begin{equation}\label{morphpent} \Rr_{12}\Rr_{13}\Rr_{23} = \Rr_{23}\Rr_{12}
\end{equation}
\textit{and $\Rr (u \otimes {\rm id}) = \Delta(u)\Rr$ for all $u\in \Ue$. Hence $\Rr$ coincides at $(\rho,\mu)\in {\rm Spec}(Z_\e)^2_{reg}$ with the evaluation map $K(\rho,\mu)$, and the regular $6j$-symbol operator $\Rr(\rho,\mu,\nu)$ is also a value of $\Rr$.} 
\end{theorem}
Note that the Pentagon equation \eqref{penteq} is equivalent to \eqref{morphpent}. By identifying $K$ with $\Rr$ each multiplicity module $M(h)$ gets a natural structure of $\Ue$-module, such that $\Rr$ is $\Ue\otimes \Ue$-linear at points $(\rho,\mu)$ where $\tau(\rho)\tau(\mu)=h$. \smallskip

Finally, let us go back to the quantum coadjoint action. By passing the coproducts on the left in \eqref{twistcop2bis}, we get two actions of $\Ue$ and $U_\e\mathfrak{h}$ on the set of bundle morphisms of $\Xi^{(2)}$, which leave $\Rr$ invariant.  By using these actions one can show that:
\begin{theorem}\label{equivariant} The $\Gg$-action on $\Xi_M$ extends to ${\rm End}(\Xi^{(2)})$ by preserving $\Rr$.
\end{theorem}
In particular, $\Rr$ is constant along the $\Delta_\Gg$-orbits in $\Xi^{(2)}$. Since the $\Gg$-orbits in ${\rm Spec}(Z_\e)\setminus \Dd$  cover the orbits of the adjoint action in $PSL_2\mc^0$, Theorem \ref{bundle} (ii) and Remark \ref{fixed} show that $\Rr$ descends to a morphism of a vector bundle of rank $n^2$ over a covering of $PSL_2\mc/\!/PSL_2\mc$ of degree $2n$. 
\subsection{Matrix dilogarithms}\label{MATDIL}  According to Theorem \ref{equivariant} one can compute $\Rr$ by restricting to pairs $(\rho,\mu)$ such that $\varphi(\rho)$, $\varphi(\mu) \in PB_+$. So we could have developed the whole theory by starting with the Heisenberg double $\Hh(U_\e b)$ of a Borel subalgebra $U_\e b \subset \Ue$. This has been done in \cite{BB1,BB2}. Explicit formulas have shown that $\Rr$ satisfies tetrahedral symmetry relations, and produced an elementary form of \eqref{morphpent} reminiscent of the five-term identities satisfied by the classical dilogarithm functions. In order to state it we need a few preparation. 
\smallskip

Define a {\it QH tetrahedron} $\Delta(b,w,f,c)$ as an oriented abstract tetrahedron $\Delta$ endowed with:
\begin{itemize}
\item a {\it branching} $b$, defined as an orientation of the edges inducing an ordering of the vertices $v_i$ by stipulating that an edge points towards the greater of its endpoints. The $2$-faces $\delta_i$ are then ordered as the opposite vertices, and the edges of $\delta_3$ are denoted by $e_j$, where the source vertices of $e_0$ and $e_1$ are $v_0$ and $v_1$, respectively.
\item a triple $w=(w_0,w_1,w_2)$, where $w_j\in \mc
\setminus \{0,1\}$ is associated to $e_j$ and the opposite edge, and $w_{j+1} = 1/(1-w_j)$; hence $w$ corresponds to the {\it cross-ratio moduli} of an ideal hyperbolic tetrahedron.
\item a {\it flattening} $f=(f_0,f_1,f_2)$ and a {\it charge} $c=(c_0,c_1,c_2)$, where $f_j$, $c_j \in \mz$ are associated to $e_j$ and the opposite edge and satisfy:
\smallskip

\noindent {\it The flattening condition:} ${\rm l}_0 + {\rm l}_1 +{\rm
    l}_2 = 0$, where
\begin{equation}\label{flatcond}
{\rm l}_j = {\rm l}_j(b,w,f)= \log(w_j) + \sqrt{-1}\pi f_j;
\end{equation}
{\it The charge condition:} 
\begin{equation}\label{chargecond}
c_0+c_1+c_2=1.\end{equation}
\end{itemize}
The branching endows $\Delta$ with a {\it $b$-orientation}, positive and denoted by $*_b=1$ when the $2$-face $\delta_3$ inherits from the orientation of its edges the opposite of the boundary orientation. In Figure \ref{fcDelta2} we show the two branched tetrahedra for $*_b=\pm 1$ (up to global symmetries), together with the $1$-skeletons of the cell decompositions dual to the interiors. 


\begin{figure}[ht]
\begin{center}
\input{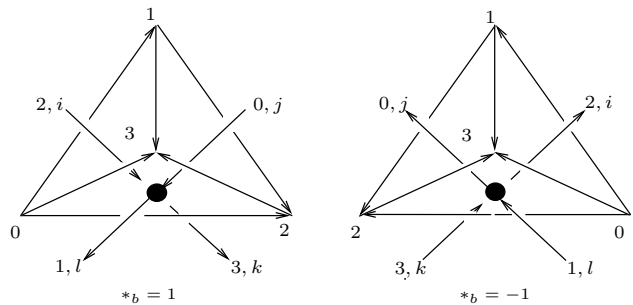}
\end{center}
\caption{\label{fcDelta2} Branched tetrahedra.}
\end{figure}

Recall that we denote by $n$ the order of $\e$. The {\it $n$th root cross-ratio moduli} of a QH tetrahedron $\Delta(b,w,f,c)$ are defined by  
\begin{equation}\label{qlb}
w_j' = \exp({\rm l}_{j,n}), \quad \mbox{where}\quad {\rm l}_{j,n} = \frac{1}{n}\left(\log(w_j) + \pi \sqrt{-1} (n+1)(f_j -*_bc_j)\right).
\end{equation}
The pairs $(w_0',w_1')\in \mc^2$ define global coordinates of the Riemann surface $\widehat{\mc}_n$ of the map  
\begin{equation}\label{maplift}
\fleche{\mc\setminus \{0,1\}}{\mc^2}{w_0}{(w_0^{\frac{1}{n}},(1-w_0)^{-\frac{1}{n}}).}
\end{equation}
Let us put $m:=(n-1)/2$, and
$$[x]:=n^{-1}\frac{1-x^n}{1-x},\quad g(x) := \prod_{j=1}^{n-1}(1 -
x\zeta^{-j})^{j/n},\quad h(x) := g(x)/g(1)$$
$$\omega(u',v'\vert n) := \prod_{j=1}^n \frac{v'}{1-u'\zeta^j},\quad
(u')^n + (v')^n = 1,\quad n \in \mn,$$ with $\omega(u',v'\vert
0) := 1$ by convention, and $x^{1/n} := \exp(\log(x)/n)$ is extended to $0^{1/n}:=0$ by continuity. The function $\omega$ is $n$-periodic in its integer argument, and $g$ is analytic
over $\mc \setminus \{r\zeta^k, r\geq 1,
k=1,\ldots,n-1\}$. \smallskip
\begin{definition}\label{matdildef} {\rm The ($n$th) {\it matrix dilogarithm} of a branched oriented tetrahedron $\Delta(b)$ is the regular map}
$$\fonc{\mathcal{R}_n(\Delta,b)}{\mz^2\times \widehat{\mc}_n}{{\rm Aut}(\mc^n\otimes \mc^n)}{(c_0,c_1,w_0',w_1')}{
\bigl((w_0')^{-c_1}(w_1')^{c_0}\bigr)^{\frac{n-1}{2}}(\Ll_n)^{*_b}(w_0',(w_1')^{-1})}$$
where 
\begin{align*} \Ll_n(u',v')_{k,l}^{i,j}& =  h(u')\
\zeta^{kj+(m+1)k^2}\ \omega(u',v'\vert i-k) \ \delta_n(i + j - l) \\
\bigl( \Ll_n(u',v')^{-1}\bigr)^{k,l}_{i,j}& =  \frac{[u']}{h(u')}\
\zeta^{-kj-(m+1)k^2}\ \frac{\delta_n(i+j-l)}{\omega(u'/\zeta,v'\vert
i-k)},
\end{align*}
\end{definition}
We will write $\mathcal{R}_n(\Delta,b)(c_0,c_1,w_0',w_1') = \Rr_n(\Delta, b, w, f,c)$. Note that the branching $b$ associates an index of $\Ll_n(w_0',w_1')^{\pm 1}$ to each $2$-face of $\Delta$ by the rule indicated in Figure \ref{fcDelta2}. 
\smallskip

Consider a move $T\ra T'$ between triangulated hexahedra $T$ and $T'$, as shown in Figure \ref{CQDidealt}. Assume that we have QH tetrahedra on both sides, having branchings that coincide at every common edge; in Figure \ref{CQDidealt} we have fixed one such global branching, but there are five other possible choices up to global symmetries.

\begin{figure}[ht]
\begin{center}
 \includegraphics[width=12cm]{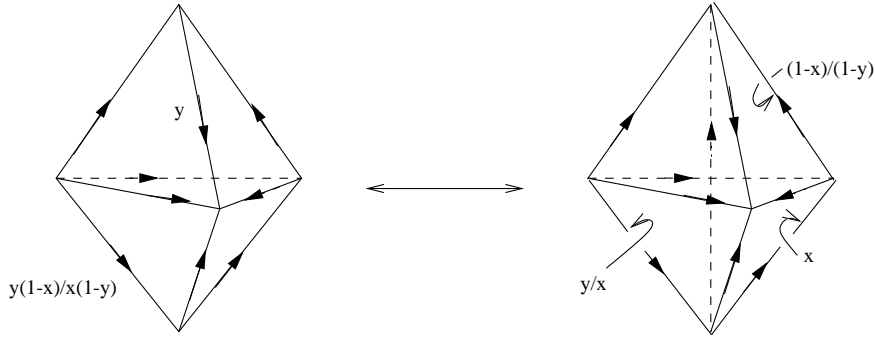}
\caption{\label{CQDidealt} An instance of transit.} 
\end{center}
\end{figure} 

Define 
\begin{equation}\label{totw} W_T'(e)=\prod_{h \rightarrow e} w'(h)^{*_{b}}\ ,\ C_T(e)= \sum_{h\rightarrow e} c(h)
\end{equation}
where:
\begin{itemize}
 \item ``$h\rightarrow e$'' means that $h$ is an edge of a QH tetrahedron of $T$ that is identified with the edge $e$ in $T$, and the products are over all such edges;
\item $*_{b}=\pm 1$ according to the $b$-orientation of the QH tetrahedron that contains $h$;
\item $w'(h)$ is the $n$th root cross-ratio modulus \eqref{qlb} at $h$, and $c(h)$ its charge.
\end{itemize}
The same notions are defined for $T'$. We say that $T\ra T'$ is a {\it $2\ra 3$ transit} if at every common edge $e$ we have 
\begin{equation}\label{preserveq}
W_T'(e)=W_{T'}'(e)\ ,\ C_T(e)=C_{T'}(e).
\end{equation}
The transit conditions are very restrictive; for instance, in Figure \ref{CQDidealt} we have shown the relations between the cross-ratio moduli $w_0$ of the five QH tetrahedra.\smallskip

To each QH tetrahedron of $T$ or $T'$,  a matrix dilogarithm $\Rr_n(\Delta, b, w, f,c)$ is associated. Define an {\it $n$-state} of $T$ or $T'$ as a function that gives every
$2$-simplex an index, with values in $\{0,\dots,n-1\}$. By the rule of Figure \ref{fcDelta2}, every $n$-state determines a matrix dilogarithm entry. As two adjacent tetrahedra induce opposite
orientations on a common face, an index is down for the matrix dilogarithm of one tetrahedron when it is up for the other. By summing over repeated indices we
get a tensor
\begin{equation}\label{tensor} \textstyle \Rr_n(T) = \sum_s \prod_{\Delta \subset T} \Rr_n(\Delta,b,w,f,c)_s
\end{equation} where the sum is over all $n$-states of $T$, and $\Rr_n(\Delta,b,w,f,c)_s$ stands
for the matrix dilogarithm entry determined by $s$. 
\smallskip

We can now state the analog of Theorem \ref{6jmain} for regular $\Ue$-modules. By comparing formulas it makes explicit the relationship between their $6j$-symbols and the matrix dilogarithms. The first claim is essentially a consequence of Theorem \ref{equivariant}. The rest is proved in \cite{BB1,BB2}. Recall the map $\varphi\colon {\rm Rep}(U_\e)\ra PSL_2\mc^0$ in \eqref{chi}, and the isomorphism $PSL_2\mc \cong {\rm Aut}(\mathbb{P}^1)$.
\begin{theorem} \label{reduction} 1) Let $(\rho,\mu,\nu)$ be a triple of $\Ue$-modules which is regular and cyclic. Put $f=\varphi(\rho)$, $g = \varphi(\mu)$, $h=\varphi(\nu)\in PSL_2\mc^0$. The $6j$-symbol operator $\Rr(\rho,\mu,\nu)$ coincides up to conjugacy with the map
$$\fonc{\Rr}{\widehat{\mc}^n}{{\rm Aut}(\mc^n\otimes \mc^n)}{(w_0',w_1')}{\Ll_n(w_0',(w_1')^{-1}),}$$
where $w_0 $ squared is the cross-ratio $ [0:f(0):fg(0):fgh(0)]$ of the points in $\mathbb{P}^1$, and the $n$th roots $w_0'$ and $w_1'$ are determined by the Casimir coordinates of $\Xi(\rho)$, $\Xi(\mu)$, $\Xi(\nu)\in {\rm Spec}(Z_\e)$. 

2) The matrix dilogarithm $\Rr_n(\Delta,b)$ satisfies:
\begin{itemize}
 \item Invariance under full tetrahedral symmetries, up to a determined projective action of $SL_2\mz$ on the source and target spaces of ${\rm Aut}(\mc^n\otimes \mc^n)$.
\item For any $2\ra 3$ transit $T\ra T'$ we have a \emph{five term identity} \begin{equation}\label{eqpentnice} \Rr_n(T) = \Rr_n(T').\end{equation}
\end{itemize}
\end{theorem}
Note that the tetrahedral symmetries of the matrix dilogarithms depend on the flattening and charge conditions \eqref{flatcond}--\eqref{chargecond}. They are necessary to get the five term identities for all the $2\ra 3$ transits.  
\smallskip

The cross-ratio $w_0$ is a complex number and is distinct from $0$ and $1$ because the triple $(\rho,\mu,\nu)$ is regular and cyclic. The $n$th root modulus $w_0'$ describes via Theorem \ref{bundle} ii) a coset of ${\rm Spec}(Z_\e)$ mod $\Gg$. The $3$-cocycloid identities \eqref{penteq} or \eqref{morphpent} coincide with \eqref{eqpentnice} exactly for the $2\ra 3$ transits $T\ra T'$ with the branching of Figure \ref{CQDidealt}. In this form, it is a non Abelian deformation of the celebrated five term identity
\begin{equation}\label{Rfivet}
{\rm L}(x) - {\rm L}(y) +{\rm L}(y/x) - {\rm L}(\frac{1-x^{-1}}{1-y^{-1}}) + 
{\rm L}(\frac{1-x}{1-y})=0
\end{equation}
which $x$, $y$ are real, $0 < y < x < 1$, and ${\rm L}$ is the dilogarithm function
\begin{equation}\label{Rdilog}
 {\rm L}(x) = -\frac{\pi^2}{6} -\frac{1}{2} \int_0^x \biggl(
\frac{\log(t)}{1-t} + \frac{\log(1-t)}{t} \biggr) \ dt.
\end{equation}
\bibliographystyle{amsalpha}

\end{document}